\documentclass{amsart}

\usepackage{amsmath,amssymb,verbatim}
\usepackage{amsthm}
\usepackage{enumerate}
\usepackage{url}
\usepackage{mathtools}

\newtheorem{theorem}{Theorem}[section]
\newtheorem{lemma}[theorem]{Lemma}
\newtheorem{corollary}[theorem]{Corollary}
\newtheorem{proposition}[theorem]{Proposition}
\newtheorem{fact}[theorem]{Fact}
\newtheorem{exercise}[theorem]{Exercise}

\theoremstyle{definition}
\newtheorem{definition}[theorem]{Definition}
\newtheorem{example}[theorem]{Example}
\newtheorem{remark}[theorem]{Remark}

\def\seq{\subseteq}

\def\inv{^{\text{-}1}}
\def\st{\textnormal{st}}
\def\smd{\raisebox{.4pt}{\textrm{\scriptsize{~\!$\triangle$\!~}}}}

\def\cA{\mathcal{A}}

\def\cB{\mathcal{B}}
\def\cC{\mathcal{C}}
\def\cF{\mathcal{F}}
\def\cH{\mathcal{H}}
\def\cN{\mathcal{N}}
\def\cP{\mathcal{P}}

\def\N{\mathbb{N}}
\def\Z{\mathbb{Z}}
\def\C{\mathbb{C}}
\def\Stab{\operatorname{Stab}}
\def\Sym{\operatorname{Sym}}
\def\GL{\operatorname{GL}}
\def\Sg{S^{\textnormal{g}}}
\def\fin{\textnormal{{f}{i}}}

\newcommand{\miff}{\makebox[.4in]{$\Leftrightarrow$}}

\AtBeginDocument{%
   \def\MR#1{}
}

\title{Stability in a group}

\author{Gabriel Conant}
\date{January 22, 2021}
\thanks{Partially supported by NSF grant DMS-1855503}

\address{DPMMS\\
University of Cambridge\\
Cambridge, CB3 0WB, UK}

\email{gconant@maths.cam.ac.uk}

\begin{document}

\begin{abstract}
We develop local stable group theory directly from topological dynamics, and extend the main tools in this subject to the setting of  stability ``in a model". Specifically, given a group $G$, we analyze the structure of sets $A\seq G$ such that the bipartite relation $xy\in A$ omits infinite half-graphs. Our proofs rely on the characterization of model-theoretic stability via Grothendieck's ``double-limit" theorem (as shown by Ben Yaacov), and the work of Ellis and Nerurkar on weakly almost periodic $G$-flows. 
\end{abstract}

\maketitle

\vspace{-5pt}

\section{Introduction}

One of the most well established and fruitful areas of model theory is the study of groups definable in \emph{stable} first-order theories, which connects mathematical logic to algebraic geometry, topological dynamics, and combinatorial number theory. At the heart of this connection is the notion of a \emph{generic} subset of a group. Given a group $G$, we say $A\seq G$ is \textbf{generic} if $G$ can be covered by finitely many left translates of $A$. An important fact in stable group theory is that the generic definable subsets of a stable group are \emph{partition regular}, i.e., the union of two non-generic definable sets is non-generic. Consequently, there exist \emph{generic types} (i.e., ultrafilters of generic definable sets), which provide a model theoretic analogue of \emph{generic points} in the sense of algebraic groups and of group actions on compact Hausdorff spaces.

Let us recall the basic definitions of stability theory (although we note that no model theory will be required for our main results). Given a complete first-order theory $T$ in a language $\mathcal{L}$, we say that an $\mathcal{L}$-formula $\varphi(x,y)$  is \textbf{stable in $T$} if for some $k\geq 1$, there is no model $M\models T$ containing tuples $a_1,\ldots,a_k,b_1,\ldots,b_k$ such that  $M\models\varphi(a_i,b_j)$ if and only if $i\leq j$ (in this case, we also say $\varphi(x,y)$ is \textbf{$k$-stable in $T$}). A theory $T$ is \textbf{stable} if every formula is stable in $T$. A \textbf{stable group} is a group whose underlying set and group operation are definable in a stable theory. 

A canonical example of a stable theory is the complete theory of an algebraically closed field. So algebraic groups provide natural examples of stable groups. The model-theoretic study of algebraic groups motivated much of the early work in stable group theory, and has now developed into an entire industry focusing on groups of ``finite Morley rank". Another important example of a stable group is any abelian group (in the pure group language). These are special cases of ``$1$-based"  groups, whose theory was developed by Hrushovski and Pillay \cite{HrPi1B}. This notion is related to the Mordell-Lang Conjecture which, in model-theoretic language, says that if $G$ is a finite rank subgroup of  a semiabelian complex variety, then the first-order structure on $G$ induced from the complex field is stable and $1$-based.  In \cite{HrMLFF}, Hrushovski combined the study of $1$-based groups and groups of finite Morley rank to prove the Mordell-Lang Conjecture for function fields in all characteristics. 

A great deal of stability theory can also be developed ``locally", i.e., for a single stable formula (see, e.g., Shelah's ``Unstable Formula Theorem" \cite[Theorem 2.2]{Shbook}). In \cite{HrPiGLF}, Hrushovski and Pillay use stable formulas to prove that any group definable in a pseudofinite field $F$ (whose complete theory is necessarily unstable) is virtually isogenous with $G(F)$, where $G$ is an algebraic group defined over $F$.   An especially spectacular display of the effectiveness of local stability is Hrushovski's work from \cite{HruAG} on the structure of approximate groups, which uses a very general ``Stabilizer Theorem" modeled after early work of Zilber on groups of finite Morley rank.

More recently, interactions with functional analysis have motivated the study of stability ``in a model". Given a first-order structure $M$, we say that a formula $\varphi(x,y)$ is \textbf{stable in $M$} if there do not exist sequences $(a_i)_{i\in I}$ and $(b_i)_{i\in I}$ from $M$, indexed by an infinite linear order $I$, such that $M\models \varphi(a_i,b_j)$ if and only if $i\leq j$. This is weaker than stability of $\varphi(x,y)$ in a theory $T$ (which is equivalent to stability in an \emph{$\aleph_0$-saturated} model of $T$). In \cite{BYGro}, Ben Yaacov established a direct connection between stability in a model and Grothendieck's characterization in \cite{GroWAP} of relatively weakly compact sets in certain Banach spaces. A natural question, which we address here, is  how this connection applies to stable group theory.

Topological dynamics has played a major role in the model theory of groups since the work of Newelski \cite{NewTD}, which provided a model-theoretic interpretation of the ``Ellis semigroup" of a \emph{$G$-flow} (i.e., a compact space with an action of  $G$ by homeomorphisms). Moreover, certain parts of a recent preprint of Hrushovski, Krupi{\'n}ski, and Pillay \cite{HKPda} indicate a close connection between stable group theory and results of Ellis and Nerurkar \cite{EllNer} on \emph{weakly almost periodic} $G$-flows.  In this paper, we will develop stable group theory entirely from  \cite{EllNer} and in the more general setting of local stability ``in a model". It is interesting to note that the original development of (global) stable group theory was roughly contemporaneous with \cite{EllNer} and related work on almost periodic minimal flows (e.g., \cite{Ausbook}, \cite{EllLTD}).

 Given a group $G$, we call a set $A\seq G$ \textbf{stable in $G$} if there do not exist sequences $(a_i)_{i\in I}$ and $(b_i)_{i\in I}$ from $G$, indexed by an infinite linear order $I$, such that $a_ib_j\in A$ if and only if $i\leq j$ (i.e., the ``formula" $xy\in A$ is stable in the structure $(G,A)$ obtained by expanding $G$ with a predicate naming $A$). One of our main results is the following structure theorem for stable subsets of groups. 

\begin{theorem}\label{thm:intro}
Let $G$ be a group, and suppose that $\cB\seq\cP(G)$ is a left-invariant Boolean algebra such that every set in $\cB$ is stable in $G$.
\begin{enumerate}[$(a)$]
\item There is a unique left-invariant  probability measure $\mu$ on $\cB$.
\item If $A\in\cB$, then $\mu(A)>0$ if and only if $A$ is generic.
\item Suppose $A\in\cB$. Then there is a finite-index subgroup $H\leq G$, which is in $\cB$, and a set $Y\seq G$, which is a union of left cosets of $H$, such that $\mu(A\smd Y)=0$. So if $\cC$ is the set of left cosets of $H$ contained in $Y$, then $\mu(A)=\mu(Y)=|\cC|/[G:H]$.
\end{enumerate}
Moreover, if $\cB$ is bi-invariant then $\mu$ is bi-invariant and is also the unique right-invariant probability measure on $\cB$; and in part $(b)$ one may choose $H$ to be a normal subgroup of $G$.
\end{theorem}

Although we have focused on left-invariant Boolean algebras in previous theorem, note that if $\cB$ is a right-invariant Boolean algebra of stable sets in $G$ then, by applying Theorem \ref{thm:intro} to the left-invariant Boolean algebra $\cB\inv=\{A\inv:A\in\cB\}$, one can obtain analogous results for $\cB$. 
Theorem \ref{thm:intro} is modeled after  \cite[Theorem 2.3]{CPT}, which focuses on the case of a single $k$-stable invariant formula. The motivation in \cite{CPT} was to prove a ``stable arithmetic regularity lemma" for finite groups, which qualitatively generalized a combinatorial result of Terry and Wolf \cite{TeWo} on $\mathbb{F}_p^n$.

To prove Theorem \ref{thm:intro}, we apply various results from \cite{EllNer} (see Theorem \ref{thm:EN}) to the action of $G$ on the \emph{Stone space} $S(\cB)$ of ``types" (or ultrafilters) over the Boolean algebra $\cB$. We will use a combinatorial formulation of the  ``fundamental theorem of local stability theory" (Theorem \ref{thm:deftypes}) to show that every function in the Ellis semigroup of $S(\cB)$ is continuous, and so $S(\cB)$ is weakly almost periodic by a result of Ellis and Nerurkar \cite{EllNer}. Moreover, $S(\cB)$  has a unique minimal closed $G$-invariant subset, which is precisely the space $\Sg(\cB)$ of \emph{generic} types in $S(\cB)$ (i.e., types containing only generic sets in $\cB$). These results form the foundation for the proof of Theorem \ref{thm:intro}.

Note that Theorem \ref{thm:intro} includes the ``global" setting where $G$ is a group definable in a stable theory and $\cB$ is the Boolean algebra of definable subsets of $G$. It also includes the case where $\cB$ is the Boolean algebra of sets $A\seq G$ that are $k$-stable for some $k\geq 1$ (i.e., $xy\in A$ is $k$-stable with respect to the theory of $(G,A)$). The ``$k$-stable case" generalizes the local setting of Hrushovski and Pillay \cite{HrPiGLF}, which applies to the Boolean algebra generated by the instances of a single left-invariant stable formula in an ambient theory $T$. This is also the setting for the weaker version of Theorem \ref{thm:intro} from \cite{CPT} mentioned above.  Thus our results extend (and in some sense complete) the work  in \cite{CPT} on $k$-stable subsets of groups. In addition to Theorem \ref{thm:intro}, we also analyze local connected components, stabilizers of generic types, and measure-stabilizers of sets, along the lines of the results obtained in \cite{CPpfNIP} for ``$k$-NIP" sets in pseudofinite groups. Among other things, we prove: 

\begin{theorem}\label{thm:prelocal}
Let $G$ be a group, and suppose $\cB\seq\cP(G)$ is a left-invariant Boolean algebra. Suppose that every set in $\cB$ is stable in $G$, and $\cB$ contains a smallest finite-index subgroup of $G$, denoted $G^0_{\!\cB}$. Given $p\in S(\cB)$, set $\Stab(p)=\{g\in G:gp=p\}$. 
\begin{enumerate}[$(a)$]
\item There is an action-preserving bijection between $\Sg(\cB)$ and the set $G/G^0_{\!\cB}$ of left cosets of $G^0_{\!\cB}$, which sends $p$ to the unique left coset of $G^0_{\!\cB}$ in $p$.
\item If $p\in \Sg(\cB)$, and $aG^0_{\!\cB}\in p$ then $\Stab(p)=aG^0_{\!\cB}a\inv$.
\end{enumerate}
Moreover, if $\cB$ is bi-invariant then $G^0_{\cB}$ is normal, the map in part $(a)$ is a group isomorphism, and $\Stab(p)=G^0_{\!\cB}$ for any $p\in \Sg(\cB)$. 
\end{theorem}

We also show that the previous theorem applies  when $\cB$ is defined from a single  left-invariant  stable relation or ``formula" (see Definitions \ref{def:RBA} and \ref{def:sif}). 

\begin{theorem}\label{thm:formula}
Let $G$ be a group, and suppose $\varphi(x,y)$ is a left-invariant stable relation on $G\times V$ for some set $V$. Let $\cB\seq\cP(G)$ be the Boolean algebra generated by all sets of the form $\{a\in G:\varphi(a,b)\}$ for all $b\in V$. Then every set in $\cB$ is stable in $G$, and $\cB$ contains a smallest finite-index subgroup of $G$.
\end{theorem}

The paper is organized as follows. In Section \ref{sec:prelim}, we recall preliminaries from topological dynamics, as well as the combinatorial formulation of definability of types and symmetry of forking for stable formulas in terms of bipartite graphs. In Section \ref{sec:SBsemi}, we establish some initial results on the Ellis semigroup of a Stone space. 

Theorems \ref{thm:intro}, \ref{thm:prelocal}, and \ref{thm:formula} are proved in Sections \ref{sec:stable} and \ref{sec:CC}. For Theorem \ref{thm:intro}, part $(a)$ is Theorem \ref{thm:left-inv}$(e)$, part $(b)$ follows from Theorem \ref{thm:gensets}$(d)$, part $(c)$ is Corollary \ref{cor:structure}, and the final statement follows from Lemma \ref{lem:must}.  Theorem \ref{thm:prelocal} is an abridged version of Theorem \ref{thm:local}, and Theorem \ref{thm:formula} appears again as Theorem \ref{thm:sif}.

 Finally, in Section \ref{sec:SAC}, we make some remarks on additive combinatorics of stable subsets of groups.

\section{Preliminaries}\label{sec:prelim}

\subsection{Topological dynamics}\label{sec:TD}

Let $G$ be a (discrete) group. In this subsection, we briefly recall the material from topological dynamics necessary for our main results. 

\begin{definition}$~$
\begin{enumerate}
\item A \textbf{$G$-flow} is a nonempty compact Hausdorff space $S$ together with a (left) action of $G$ on $S$ by homeomorphisms.
\item Given a $G$-flow $S$, a \textbf{subflow} is a nonempty closed $G$-invariant subset of $S$.
\item A $G$-flow is \textbf{minimal} if it has no proper subflows.
\item  Let $S$ be a $G$-flow, and consider $S^S$ with the product topology.
\begin{enumerate}[$(i)$]
\item Given $a\in G$, define $\pi_a\colon S\to S$ such that $\pi_a(p)=ap$.
\item Define $E(S)$ to be the closure of $\{\pi_a:a\in G\}$ in $S^S$. 
\end{enumerate}
\end{enumerate}
\end{definition}

\begin{proposition}[Ellis]\label{prop:Ellis}
Let $S$ be a $G$-flow. Then $E(S)$ is a semigroup under composition of functions, and a $G$-flow under the action $(a,f)\mapsto \pi_a\circ f$.
\end{proposition}
\begin{proof}
This is an exercise (or see, e.g., \cite[Proposition 3.2]{EllLTD}, \cite[Theorem 2.29]{HiStbook}).
\end{proof}

\begin{definition}
The \textbf{Ellis semigroup} of a $G$-flow $S$ is $(E(S),\circ)$. 
\end{definition}

\begin{definition}
Let $S$ be a $G$-flow and let $\cC(S)$ be the space of continuous complex-valued functions on $S$. A function $f\in\cC(S)$ is \textbf{weakly almost periodic} if the set $\{f^a:a\in G\}$ is relatively compact in the weak topology on $\cC(S)$ where, given $a\in G$ and $p\in S$, $f^a(p):=f(ap)$. The flow $S$ is \textbf{weakly almost periodic} if every $f\in \cC(S)$ is weakly almost periodic.
\end{definition}

Although the previous definition is rooted in functional analysis, we will not need to delve further into these underlying notions, thanks to the following theorem.

\begin{theorem}[Ellis-Nerurkar \cite{EllNer}]\label{thm:EN}
Let $S$ be a $G$-flow.
\begin{enumerate}[$(a)$]
\item $S$ is weakly almost periodic if and only if every $f\in E(S)$ is continuous.
\item If $S$ is weakly almost periodic then:
\begin{enumerate}[$(i)$]
\item $E(S)$ has a unique minimal subflow $C\seq E(S)$;
\item $C$ contains a unique idempotent $u$, and $f\circ u=u\circ f$ for any $f\in E(S)$;
\item $(C,\circ)$ is a compact group with identity $u$.
\item If $S$ has a unique minimal subflow then there is a unique $G$-invariant Borel probability measure on $S$ (i.e., $S$ is \textbf{uniquely ergodic}).
\end{enumerate}
\end{enumerate}
\end{theorem}
\begin{proof}
See Propositions II.2, II.5, and II.10 of \cite{EllNer}.
\end{proof}

\begin{definition}\label{def:genpoint}
Let $S$ be a $G$-flow.
\begin{enumerate}
\item A set $X\seq S$ is \textbf{generic} if $S=FX$ for some finite $F\seq G$.
\item A point $p\in S$ is \textbf{generic} if every open set containing $p$ is generic.
\end{enumerate}
\end{definition}

\begin{proposition}[Newelski \cite{NewTD}]\label{prop:New}
Let $S$ be a $G$-flow. The following are equivalent.
\begin{enumerate}[$(i)$]
\item $S$ has a unique minimal subflow.
\item There is a generic point in $S$.
\item The set of generic points in $S$ is the unique minimal subflow of $S$.
\item Every generic open set in $S$ contains a generic point.
\end{enumerate}
\end{proposition}
\begin{proof}
$(i)\Rightarrow (ii)$ and $(iii)\Rightarrow (iv)$ follow from \cite[Lemma 1.7]{NewTD}. $(ii)\Rightarrow (iii)$ is \cite[Corollary 1.9]{NewTD}. $(iii)\Rightarrow (i)$ and $(iv)\Rightarrow (ii)$ are trivial. 
\end{proof}

\subsection{Stone spaces and generic types}

\begin{definition}\label{def:Stone}
Let $U$ be a set, and let $\cB\seq\cP(U)$ be a Boolean algebra. A subset $p\seq \cB$ is a \textbf{$\cB$-type} if $\emptyset\not\in p$, $p$ is closed under  finite intersections, and for any $A\in\cB$ either $A\in p$ or $U\backslash A\in p$. The \textbf{Stone space of $\cB$}, denoted $S(\cB)$, is the set of all $\cB$-types. Given $A\in\cB$, define $S_{\!A}(\cB):=\{p\in S(\cB):A\in p\}$. 
\end{definition}

In the setting of the previous definition, if $\cB=\cP(U)$ then $\cB$-type is also called an \emph{ultrafilter on $U$}. So one can think of $\cB$-types as ultrafilters relativized to an arbitrary Boolean algebra $\cB$. A ``trivial" example of a $\cB$-type is $p^{\cB}_a:=\{A\in\cB:a\in A\}$, where $a\in U$, which we call the \textbf{principal $\cB$-type on $a$}.

We now state some basic exercises, which justify our use of the word ``space" in Definition \ref{def:Stone}, and  connect Stone spaces over groups to topological dynamics. 

\begin{exercise}\label{exc:Stone}
Let $U$ be a set, and let $\cB\seq\cP(U)$ be a Boolean algebra.
\begin{enumerate}[$(a)$]
\item $S(\cB)$ is a totally disconnected compact Hausdorff space under the topology with basic open sets of the form $S_{\!A}(\cB)$ for all $A\in\cB$. 
\item A subset of $S(\cB)$ is clopen if and only if it is of the form $S_{\!A}(\cB)$ for some $A\in\cB$. 
\item The set $\{p^{\cB}_a:a\in U\}$ of principal $\cB$-types is dense in $S(\cB)$. 
\end{enumerate}
\end{exercise}

A standard fact is that a topological space is profinite (i.e., a projective limit of finite discrete spaces) if and only if it is compact, Hausdorff, and totally disconnected \cite[Theorem 2.1.3]{RZbook}. Thus the Stone space of a Boolean algebra is profinite. 

\begin{exercise}
 Let $G$ be a group, and suppose $\cB\seq\cP(G)$ is a left-invariant Boolean algebra. If $p\in S(\cB)$ and $g\in G$, then $gp:=\{gA:A\in p\}$ is in $S(\cB)$. Moreover, $S(\cB)$ is a $G$-flow under the action $(g,p)\mapsto gp$.
\end{exercise} 

Note that in the previous exercise, if $\cB$ is instead \emph{right-invariant}, then we have a corresponding right action $pg=\{Ag:g\in G\}$. Although our focus is on left-invariance, we will also use this right action on a few occasions.

Recall that in Definition \ref{def:genpoint}, we defined generic subsets and points in $G$-flows. The next definition defines genericity for subsets of groups, and the subsequent exercise explains the connection to  flows arising from Stone spaces.

\begin{definition}
A subset $A$ of a group $G$ is \textbf{generic} if $G$ can be covered by finitely many left translates of $A$, i.e., $G=FA$ for some finite $F\seq G$. 
\end{definition}

\begin{exercise}\label{exc:gen}
Let $G$ be a group, and let $\cB\seq\cP(G)$ be a left-invariant Boolean algebra. Then a set $A\in\cB$ is generic if and only if $S_{\!A}(\cB)$ is a generic subset of $S(\cB)$. Moreover, a type $p\in S(\cB)$ is generic if and only if every $A\in p$ is generic. 
\end{exercise}

\begin{definition}
Let $G$ be a group, and suppose $\cB\seq\cP(G)$ is a left-invariant Boolean algebra. Define $\Sg(\cB)$ to be the set of generic $\cB$-types. Given $A\in\cB$, define $\Sg_{\!A}(\cB):=\{p\in \Sg(\cB):A\in p\}$. 
\end{definition}

Let $G$ be a group, and let $\cB\seq\cP(G)$ be a left-invariant Boolean algebra. Then $\Sg(\cB)$ is a closed subset of $S(\cB)$ since, if $p\in S(\cB)$ is not generic then, by choosing some non-generic $A\in p$, we obtain an open neighborhood $S_{\!A}(\cB)$ of $p$ disjoint from $\Sg(\cB)$. In the subsequent work, we will focus on $\Sg(\cB)$ as a topological space in its own right. So we point out that $\Sg(\cB)$ is a profinite space with a basis consisting of the clopen sets $\Sg_{\!A}(\cB)$ for all $A\in\cB$ (note that $\Sg_{\!A}(\cB)=S_{\!A}(\cB)\cap\Sg(\cB)$).

It will be helpful for later results to have the following restatement of Proposition \ref{prop:New} in the setting of Stone spaces.

\begin{corollary}\label{cor:NewSB}
Let $G$ be a group and suppose $\cB$ is a left-invariant Boolean algebra of subsets of $G$. The following are equivalent.
\begin{enumerate}[$(i)$]
\item $S(\cB)$ has a unique minimal subflow.
\item $\Sg(\cB)$ is nonempty.
\item $\Sg(\cB)$ is the unique minimal subflow of $S(\cB)$.
\item If $A\in \cB$ is generic then $\Sg_{\!A}(\cB)$ is nonempty. 
\item If $\cF\seq\cB$ is closed under finite intersections and contains only generic sets, then there is some $p\in \Sg(\cB)$ such that $\cF\seq p$.
\end{enumerate}
\end{corollary}
\begin{proof}
By Exercise \ref{exc:gen}, properties $(i)$ through $(iv)$ are translations of the corresponding properties in Proposition \ref{prop:New}. Clearly $(v)\Rightarrow(iv)$. For $(iv)\Rightarrow (v)$ note that  $\{\Sg_{\!A}(\cB):A\in\cF\}$ is a collection of nonempty closed subsets of $\Sg(\cB)$ with the finite intersection property. So the result follows from compactness of $\Sg(\cB)$. 
\end{proof}

\subsection{Stable relations}\label{sec:pre-stable}

Let $U$ and $V$ be fixed sets, and fix a binary relation $\varphi(x,y)$ on $U\times V$. Given $(a,b)\in U\times V$, we write $\varphi(a,b)$ to denote that the relation holds on $(a,b)$. If $b\in V$ then $\varphi(U,b)$ denotes the fiber $\{a\in U:\varphi(a,b)\}$. Dually, if $a\in U$ then $\varphi(a,V)=\{b\in V:\varphi(a,b)\}$. Combinatorially, one can view $\varphi(x,y)$ as a bipartite graph on $U\times V$, and the fibers of $\varphi(x,y)$ as edge neighborhoods. 

\begin{definition}\label{def:RBA}
$~$
\begin{enumerate}
\item Let $I$ be a linearly ordered set. Then $\varphi(x,y)$ \textbf{codes $I$} if there are sequences $(a_i)_{i\in I}$ in $U$ and $(b_i)_{i\in I}$ in $V$, such that $\varphi(a_i,b_j)$ if and only if $i\leq j$.
\item $\varphi(x,y)$ is \textbf{stable} if it does not code an infinite linear order.
\item Given $k\geq 1$, $\varphi(x,y)$ is \textbf{$k$-stable} if it does not code a linear order of size $k$.
\end{enumerate}
\end{definition}

\begin{remark}
In the model-theoretic setting, $U$ and $V$ would typically be sorts in some first-order structure $M$ (e.g., $U=V=M$), and $\varphi(x,y)$ would be a first-order formula. In this case, our definition of stability for $\varphi(x,y)$ is referred to as stability ``in a model". If $M$ is $\aleph_0$-saturated, then $\varphi(x,y)$ is stable (as defined above) if and only if it is $k$-stable for some $k\geq 1$, and this can be expressed as a first-order property of the theory of $M$. On the other hand, if $M$ is not $\aleph_0$-saturated, then there may be formulas that are stable but not $k$-stable for any $k\geq 1$. Thus stability in a model is more general than what is usually considered in model-theoretic literature. So we caution the reader  that when we use the word ``stable" in this paper, we will mean the weaker notion of stability in a model.
\end{remark}

Next we use the fibers of $\varphi(x,y)$ to define Boolean algebras on $U$ and $V$.

\begin{definition}\label{def:dphi}$~$
\begin{enumerate}
\item Define $\cB_\varphi$ to be the Boolean algebra  generated by $\{\varphi(U,b):b\in V\}$.
\item Define $\cB^*_\varphi$ to be the Boolean algebra generated by $\{\varphi(a,V):a\in U\}$.
\item Given a type $p\in S(\cB_\varphi)$, set $d_\varphi p = \{b\in V:\varphi(U,b)\in p\}$.
\item Given a type $p\in S(\cB^*_\varphi)$, set $d^*_\varphi p =\{a\in U:\varphi(a,V)\in p\}$. 
\end{enumerate}
\end{definition}

Note that if $p\in S(\cB_\varphi)$, then the set $d_\varphi p$ uniquely determines $p$ by definition of $\cB_\varphi$.
This notation draws from the model-theoretic notion of type definitions.  One of the fundamental results of model-theoretic stability theory is that, when $\varphi(x,y)$ is stable, the sets $d_\varphi p$ are themselves ``definable" using $\varphi(x,y)$. We now state this result precisely, along with a second important fact related to the model-theoretic notion of ``forking" (which we will not need to discuss here).

\begin{theorem}\label{thm:deftypes}
Suppose $\varphi(x,y)$ is stable, and fix $p\in S(\cB_\varphi)$ and $q\in S(\cB^*_{\varphi})$.
\begin{enumerate}[$(a)$]
\item \textnormal{(definability of types)} $d_\varphi p\in \cB^*_{\varphi}$ and $d^*_{\varphi}q\in\cB_{\varphi}$.
\item \textnormal{(symmetry of forking)} $d_\varphi p\in q$ if and only if $d^*_{\varphi}q\in p$.
\end{enumerate}
\end{theorem}

In the model-theoretic context, part $(a)$ of Theorem \ref{thm:deftypes}  is evident from work of Pillay \cite{PilDTH}. However, in \cite{BYGro}, Ben Yaacov proves Theorem \ref{thm:deftypes}$(a)$ as a corollary of Grothendieck's  characterization in \cite{GroWAP} of relatively weakly compact sets in the Banach space of bounded continuous functions on some fixed topological space. See also \cite{PilGro} and \cite{StarGro} for expositions of this result. Part $(b)$ of Theorem \ref{thm:deftypes} follows easily from the Grothendieck approach to part $(a)$ (see \cite{StarGro}). It is also worth noting that Grothendieck's work is a key ingredient in the proof of Theorem \ref{thm:EN}$(a)$. 

When $\varphi(x,y)$ is $k$-stable for some $k\geq 1$,  part $(a)$ of Theorem \ref{thm:deftypes} was proved by Shelah (see \cite[Theorem II.2.2]{Shbook}) and, given part $(a)$, part $(b)$ is not hard to prove directly (see, e.g., \cite[Lemma 5.7]{HrPiGLF}).

We have now reviewed all of the preliminaries needed to prove Theorems \ref{thm:intro} and \ref{thm:prelocal}. For Theorem \ref{thm:formula}, we will also need a result on measures. Given a Boolean algebra $\cB$ (on some set $U$), let $M(\cB)$ denote the compact Hausdorff space of \textbf{probability measures} on $\cB$ (i.e., finitely additive functions $\mu\colon\cB\to[0,1]$ with $\mu(U)=1$), with the subspace topology from $[0,1]^{\cB}$. We may view $S(\cB)$ as a closed set in $M(\cB)$ by identifying $p\in S(\cB)$ with a $\{0,1\}$-valued measure. A well-known result of Keisler \cite{Keis} is that if $\varphi(x,y)$ is $k$-stable, then any $\mu\in M(\cB_\varphi)$ is a weighted sum of countably many $\cB_\varphi$-types  (see also \cite[Fact 1.1]{PiDR}). In \cite{GanStab}, Gannon uses Theorem \ref{thm:deftypes}$(a)$, together with the Sobczyk-Hammer Decomposition Theorem  from measure theory, to  generalize this to the case that $\varphi(x,y)$ is stable. 

\begin{theorem}\label{thm:summeas}
Suppose $\varphi(x,y)$ is stable and $\mu\in M(\cB_\varphi)$. Then there are $p_n\in S(\cB_\varphi)$ and $\alpha_n\in [0,1]$, for $n\in\N$, such that $\mu=\sum_{n=0}^\infty\alpha_n p_n$.
\end{theorem}

\section{Semigroups on Stone spaces}\label{sec:SBsemi}
Throughout this section, we let $G$ be a fixed group. 
The goal of this section is to formulate assumptions on a left-invariant Boolean algebra $\cB\seq\cP(G)$ under which $E(S(\cB))$ is isomorphic to a more natural semigroup. So let us first precisely define the kind of semigroup we are interested in.

\begin{definition}
A \textbf{$G$-semigroup} is a $G$-flow $S$ equipped with a semigroup operation $\cdot$, so that  if $g\in G$ and $p,q\in S$ then $g(p\cdot q)=(gp)\cdot q$.
\end{definition}

\begin{example}
If $S$ is a $G$-flow then $(E(S),\circ)$ is a $G$-semigroup by Proposition \ref{prop:Ellis}.
\end{example}

Our aim is to give a more natural description of $E(S(\cB))$ as a $G$-semigroup, under certain assumptions on $\cB$. To motivate these assumptions, we first recall a well-known example of a Stone space with a canonical $G$-semigroup structure.  

\begin{example}\label{ex:StoneCech} 
The Stone space $S(\mathcal{P}(G))$  is denoted $\beta G$ and is called the \emph{Stone-\v{C}ech compactification} of $G$. Given $p,q\in \beta G$, set
\[
p \ast q=\big\{A\seq G:\{x\in G:x\inv A\in q\}\in p\big\}.
\]
Then $\ast$ is a well-defined semigroup operation on $\beta G$ and, moreover, $(\beta G,\ast)$ is a $G$-semigroup under the usual action of $G$. We will prove a generalization of this fact in Proposition \ref{prop:SBsemi} below. See also \cite[Section 4.1]{HiStbook} for further details. 
\end{example}

Toward adapting the previous example to more general situations, let us first redefine $\ast$ for an arbitrary left-invariant Boolean algebra.

\begin{definition}\label{def:dpA}
Let $\cB\seq\cP(G)$ be a left-invariant Boolean algebra. 
\begin{enumerate}
\item Given $p\in S(\cB)$ and $A\in\cB$, define $dp(A)=\{x\in G:x\inv A\in p\}$. 
\item Given $p,q\in S(\cB)$, define $p\ast q=\{A\in\cB:dq(A)\in p\}$. 
\end{enumerate}
\end{definition}

The set $dp(A)$ above can be connected to the ``type definitions" discussed in Section \ref{sec:pre-stable}. In particular, let $\cB\seq\cP(G)$ be the Boolean algebra generated by the left translates of a fixed set $A\seq G$, and let $\varphi(x,y)$ be the relation $x\in y\inv A$ on $G\times G$. 
Then $\cB=\cB_{\varphi}$ and, given $p\in S(\cB)$, we have $d_\varphi p=dp(A)$. 

 \begin{remark}\label{rem:BAhom}
In the context of Definition \ref{def:dpA}, we view $dp$ as a map $dp\colon \cB\to\cP(G)$. As a good warm-up exercise, the reader should verify that $dp$ is a left-invariant homomorphism of Boolean algebras.
\end{remark}

Note that if $\cB\seq\cP(G)$ is a left-invariant Boolean algebra, and $p,q\in S(\cB)$, then $p\ast q$ is only defined to be a \emph{subset} of $\cP(\cB)$. There is no reason to expect that $p\ast q$ is a $\cB$-type, let alone that $\ast$ determines a semigroup operation on $S(\cB)$. Indeed, given $A\in\cB$, the set $dq(A)$ may not even be in $\cB$. However, if we impose this assumption, then we can recover a semigroup structure on $S(\cB)$ as in the case of $\beta G$.

\begin{proposition}\label{prop:SBsemi}
Let $\cB\seq\cP(G)$ be a left-invariant Boolean algebra such that $dp(A)\in\cB$ for all $p\in S(\cB)$ and $A\in\cB$. Then $(S(\cB),\ast)$ is a $G$-semigroup.
\end{proposition}
\begin{proof}
Given our assumptions on $\cB$, it follows easily from Remark \ref{rem:BAhom} that $S(\cB)$ is closed under $\ast$. To prove associativity, fix $p,q,r\in S(\cB)$ and $A\in\cB$. Then
\begin{multline*}
d(q\ast r)(A)=\{x\in G:x\inv A\in q\ast r\}=\{x\in G:dr(x\inv A)\in q\}\\
=\{x\in G:x\inv dr(A)\in q\}=dq(dr(A)).
\end{multline*}
Therefore
\begin{multline*}
A\in p\ast (q\ast r)\miff d(q\ast r)(A)\in p\miff dq(dr(A))\in p\\
\miff dr(A)\in p\ast q\miff A\in (p\ast q)\ast r.
\end{multline*}
Finally, given $p,q\in S(\cB)$ and $g\in G$, we have $g(p\ast q)=gp\ast q$ by similar calculations (in particular, the fact that $dq(g\inv A)=g\inv dq(A)$ for any $A\in\cB$). 
\end{proof}

We will soon see that in the setting of the previous proposition, $(S(\cB),\ast)$ and $(E(S(\cB)),\circ)$ are isomorphic as $G$-semigroups.
So a natural question is how easily one can find Boolean algebras satisfying the hypotheses of this result. In light of the discussion after Definition \ref{def:dpA}, Theorem \ref{thm:deftypes}$(a)$ looks promising for the case of Boolean algebras of \emph{stable} sets (see Definition \ref{def:stableG}), which will be our main focus. However, there is one concrete obstacle. In particular, suppose $\cB\seq\cP(G)$ is a left-invariant Boolean algebra. Given $A\in\cB$ and $g\in G$, if $p=p^{\cB}_{g}$ is the principal $\cB$-type on $g$, then $dp(A)=Ag\inv$. Therefore, any left-invariant Boolean algebra satisfying the hypotheses of Proposition \ref{prop:SBsemi} is automatically bi-invariant. Let us record this conclusion, along with some other basic observations (left to the reader).

\begin{proposition}\label{prop:pggp}
Let $G$ be a group, and suppose $\cB\seq\cP(G)$ is a left-invariant Boolean algebra such that $dp(A)\in\cB$ for all $p\in S(\cB)$ and $A\in\cB$. Then $\cB$ is bi-invariant. Moreover, if $p\in S(\cB)$ and $g\in G$, then $p^{\cB}_g\ast p=gp$ and $p\ast p^{\cB}_g=pg$.
\end{proposition}

Our task now is to find a weaker version of the assumption in Proposition \ref{prop:SBsemi}, which does not force $\cB$ to  be bi-invariant, but still leads to control of $E(S(\cB))$. To motivate this investigation, we first consider an example.

\begin{example}\label{ex:easy}
Fix a finite-index subgroup $H\leq G$, and let $\cB$ be the Boolean algebra generated by all left cosets of $H$. Then a subset of $G$ is in $\cB$ if and only if it is a union of such cosets. Also, $\cB$ is bi-invariant if and only if $H$ is normal. Indeed, if $H$ is not normal then some conjugate $aHa\inv$ does not contain $H$. So $aHa\inv$ is not in $\cB$ since any set in $\cB$ containing the identity must contain $H$. 

Now consider the Stone space $S(\cB)$. From basic properties of types, one can see that any $p\in S(\cB)$ contains a unique left coset of $H$, which completely determines $p$. So $S(\cB)$ is in bijection with the set $X$ of left cosets of $H$. 
Since $S(\cB)$ is finite, $E(S(\cB))=\{\pi_a:a\in G\}$ where $\pi_a\colon p\mapsto ap$. Altogether, $E(S(\cB))$ isomorphic to the image of $G$ under the left regular representation in $\textnormal{Sym}(X)$, and so $E(S(\cB))$ is the \emph{group} $G/K$, where $K=\bigcap_{g\in G}gHg\inv$ is kernel of this representation. The appearance of $K$ can also be predicted by analyzing the maps $dp$ for $p\in S(\cB)$. In particular, given $a,b\in G$, if $p\in S(\cB)$ is the unique type containing $bH$, then $dp(aH)=aHb\inv$. So $dp(\cB)$ is contained in  the Boolean algebra generated by the cosets of $N$,  which is also the smallest bi-invariant Boolean algebra containing $\cB$. \end{example}

\begin{definition}
Given a left-invariant Boolean algebra $\cB\seq\cP(G)$, define $\cB^{\sharp}$ to be the smallest bi-invariant Boolean algebra containing $\cB$.
\end{definition}

Motivated by Example \ref{ex:easy}, we now study left-invariant Boolean algebras $\cB\seq G$ such that $dp(\cB)\seq\cB^\sharp$ for all $p\in S(\cB)$. This will be a natural weakening of the assumption in Proposition \ref{prop:SBsemi} suitable for left-invariant Boolean algebras that are not necessarily bi-invariant.

\begin{lemma}\label{lem:ESli}
Suppose $\cB\seq\cP(G)$ is a left-invariant Boolean algebra such that $dp(A)\in \cB^\sharp$ for all $A\in\cB$ and $p\in S(\cB)$. 
\begin{enumerate}[$(a)$]
\item If $A\in\cB^\sharp$ and $p\in S(\cB^\sharp)$ then $dp(A)\in\cB^\sharp$, and so $(S(\cB^\sharp),\ast)$ is a $G$-semigroup.
\item Given $p\in S(\cB^\sharp)$ and $q\in S(\cB)$, define $f_p(q)= \{A\in\cB:dq(A)\in p\}$. Then $f_p\in S(\cB)^{S(\cB)}$ for all $p\in S(\cB^\sharp)$.
\end{enumerate}
\end{lemma}
\begin{proof}
Part $(a)$. Note that the second claim follows from the first and Proposition \ref{prop:SBsemi}. So fix $p\in S(\cB^\sharp)$ and $A\in\cB^\sharp$. We want to show $dp(A)\in\cB^\sharp$.  By Remark \ref{rem:BAhom}, we may assume $A$ is of the form $Bg$ for some $B\in\cB$ and $g\in G$. Let $q=pg\inv{\upharpoonright}\cB\in S(\cB)$. Then $dp(A)=dq(B)$, and $dq(B)\in\cB^\sharp$ by assumption on $\cB$. 

Part $(b)$.  Note the similarity between $f_p(q)$, as defined here for $p\in S(\cB^\sharp)$ and $q\in S(\cB)$, and $p\ast q$ for $p,q\in S(\cB)$ as in Definition \ref{def:dpA}. With this in mind, the verification that $f_p(q)\in S(\cB)$ for any $q\in S(\cB)$ is essentially the same as the first part of the proof of Proposition \ref{prop:SBsemi}. 
\end{proof}

We can now prove the main result of this section. 

\begin{theorem}\label{thm:ESli}
Suppose $\cB\seq\cP(G)$ is a left-invariant Boolean algebra such that $dp(A)\in \cB^\sharp$ for all $A\in\cB$ and $p\in S(\cB)$. Then the map $\Phi:p\mapsto f_p$, where $f_p$ is as in Lemma \ref{lem:ESli}, is a $G$-semigroup isomorphism between $(S(\cB^\sharp),\ast)$ and $(E(S(\cB)),\circ)$
\end{theorem}
\begin{proof}
By Lemma \ref{lem:ESli}$(b)$, $\Phi$ is a map from $S(\cB^\sharp)$ to $S(\cB)^{S(\cB)}$. We first show $\Phi$ is continuous. Fix $q\in S(\cB)$ and $A\in\cB$, and let $U=\{f\in S(\cB)^{S(\cB)}:A\in f(q)\}$ be the corresponding sub-basic open set in $S(\cB)^{S(\cB)}$. Then $f_p\in U$ if and only if $dq(A)\in p$, and so $\Phi\inv(U)=S_{dq(A)}(\cB^\sharp)$, which is open in $S(\cB^\sharp)$.

Recall that $E(S(\cB))$ is the closure in $S(\cB)^{S(\cB)}$ of $\{\pi_a:a\in G\}$ where $\pi_a\colon p\mapsto ap$. 
One easily checks that $\Phi(p^{\cB^\sharp}_a)=\pi_a$ for any $a\in G$. Since $\{p_a^{\cB^\sharp}:a\in G\}$ is dense in $S(\cB^\sharp)$ and $\Phi$ is continuous, it follows that $\{\pi_a:a\in G\}$ is a dense subset of $\Phi(S(\cB^\sharp))$. Therefore $\Phi(S(\cB^\sharp))\seq E(S(\cB))$. Since $\Phi$ is continuous, and $S(\cB^\sharp)$ is compact, it follows that $\Phi(S(\cB^\sharp))$ is compact, and hence closed in $S(\cB)^{S(\cB)}$. Altogether, $\Phi(S(\cB^\sharp))=E(S(\cB))$.

From now on, we view $\Phi$ as a surjective function from $S(\cB^\sharp)$ to $E(S(\cB))$. To show $\Phi$ is injective, fix $p,q\in S(\cB^\sharp)$ such that $f_p=f_q$. To show that $p=q$, it suffices to fix $A\in\cB$ and $g\in G$, and show $Ag\in p$ if and only if $Ag\in q$. To see this, let $r=p^{\cB}_{g\inv}$, and note that $Ag=dr(A)$. Therefore
\begin{multline*}
Ag\in p\miff dr(A)\in p\miff A\in f_p(r)\\
\miff A\in f_q(r)\miff dr(A)\in q\miff Ag\in q.
\end{multline*}

Since $\Phi$ is a continuous bijection between compact Hausdorff spaces, it is a homeomorphism. So to show that $\Phi$ is an isomorphism of $G$-flows, we just need to check that it preserves the actions of $G$ on $S(\cB^\sharp)$ and $E(S(\cB))$, i.e., $\Phi(ap)=\pi_a\circ \Phi(p)$ for any $a\in G$ and $p\in S(\cB^\sharp)$. So fix $a\in G$ and $p\in S(\cB^\sharp)$. Then, given $q\in S(\cB)$, we have that for any $A\in\cB$,
\[
A\in f_{ap}(q)\miff a\inv dq(A)\in p\miff dq(a\inv A)\in p\miff A\in af_p(q),
\]
and so $f_{ap}(q)=af_p(q)=\pi_a(f_p(q))$. This shows $\Phi(ap)=\pi_a\circ\Phi(p)$.

Finally, we show that $\Phi$ preserves the semigroup operations on $S(\cB^\sharp)$ and $E(S(\cB))$. Fix $p,q\in S(\cB^\sharp)$. For any $r\in S(\cB)$ and $A\in\cB$, we have $dq(dr(A))=d(f_q(r))(A)$ (similar to the proof of Proposition \ref{prop:SBsemi}) and so
\[
A\in f_{p\ast q}(r)\miff dr(A)\in p\ast q\miff d(f_q(r))(A)\in p\miff A\in f_p(f_q(r)).
\]
Therefore $\Phi(p\ast q)(r)=(\Phi(p)\circ\Phi(q))(r)$ for any $r\in S(\cB)$.
\end{proof}

\section{Stable subsets of groups}\label{sec:stable}

Throughout this section, we work with a fixed group $G$. 

\subsection{Boolean algebras of stable sets}
\begin{definition}\label{def:stableG}
A set $A\seq G$ is \textbf{stable} (in $G$) if the relation $xy\in A$ on $G\times G$ is stable (see Definition \ref{def:RBA}). Define $\cB^{\st}_G$ to be the collection of stable subsets of $G$.
\end{definition}

The following fact is well-known and left as an exercise. 

\begin{fact}\label{fact:BSTG}
 $\cB^{\st}_G$ is a bi-invariant Boolean algebra, which contains all subgroups of $G$ and is closed under $A\mapsto A\inv$. 
\end{fact}

The primary goal of this paper is Theorem \ref{thm:intro}, which gives a structure theorem for arbitrary left-invariant sub-algebras of $\cB^{\st}_G$. The main reason we consider sub-algebras of $\cB^{\st}_G$, rather than just working with $\cB^{\st}_G$ itself, is due to issues with ``definability". In particular, Theorem \ref{thm:intro} shows that if $\cB$ is a left-invariant sub-algebra of $\cB^{\st}_G$, then any set in $\cB$ can be approximated by cosets of a finite-index subgroup of $G$, \emph{which is also in $\cB$}. So a stable subset of $G$ can be approximated by a subgroup that has some connection to original set. This kind of control is important in applications. For example,  \cite{CPT} uses stable subsets of pseudofinite groups to prove results about stable subsets of finite groups. In order for this to work, one needs restrict to the Boolean algebra of \emph{internal} stable subsets of a pseudofinite group (e.g., so that an internal stable set is approximated by an internal subgroup). 

We will also focus on the general setting of \emph{left-invariant} sub-algebras of $\cB^{\st}_G$, despite the fact that bi-invariant Boolean algebras are easier to work with. In addition to an overall motivation for more general results, the main reason for this focus is to capture the right notion of ``local stability theory" as formulated in the model-theoretic literature and, in particular, the work of Hrushovski and Pillay \cite[Section 5]{HrPiGLF}. For example, unlike the ``global" setting groups definable in stable theories, the ``bi-stratification" of a left-invariant stable formula need not be stable (see Example \ref{ex:bad}). A full analysis of this setting will be done in Section \ref{sec:CC}.  

Let us now start the journey toward our main results.
We first establish some properties of left-invariant sub-algebras of $\cB^{\st}_G$. For instance, we show that they satisfy the assumptions of Theorem \ref{thm:ESli}, and they also behave nicely with respect to a dual version of the $dp$ map.

\begin{lemma}\label{lem:BAstab}
Suppose $\cB$ is a left-invariant sub-algebra of $\cB^{\st}_G$. Fix $A\in\cB$ and, given $p\in S(\cB^\sharp)$, set $dp^*(A)=\{x\in G:Ax\inv\in p\}$.
\begin{enumerate}[$(a)$]
\item If $p\in S(\cB)$ then $dp(A)\in\cB^\sharp$.
\item If $p\in S(\cB^\sharp)$ then $dp^*(A)\in \cB$.
\item If $p\in S(\cB^\sharp)$ and $q\in S(\cB)$, then $dq(A)\in p$ if and only if $dp^*(A)\in q$.
\end{enumerate}
\end{lemma}
\begin{proof}
Let $\varphi(x,y)$ be the relation $x\in y\inv A$ on $G\times G$. Then $\varphi(x,y)$ is stable, $\cB_{\varphi}$ is the Boolean algebra generated by $\{gA:g\in G\}$, and $\cB^*_{\varphi}$ is the Boolean algebra generated by $\{Ag:g\in G\}$. In particular, $\cB_{\varphi}\seq\cB$ and $\cB^*_{\varphi}\seq\cB^\sharp$.

Part $(a)$. Fix $p\in S(\cB)$ and set $p_0=p{\upharpoonright}\cB_\varphi$. Then $dp(A)=d_\varphi p_0$, and so $dp(A)\in\cB^*_{\varphi}\seq\cB^\sharp$ by Theorem \ref{thm:deftypes}$(a)$.

Part $(b)$. Fix $p\in S(\cB^\sharp)$ and set $p_0=p{\upharpoonright}\cB^*_{\varphi}$. Then $dp^*(A)=d^*_{\varphi}p_0$, and so $dp^*(A)\in\cB_{\varphi}\seq\cB$ by Theorem \ref{thm:deftypes}$(a)$.

Part $(c)$. Fix $p\in S(\cB^\sharp)$ and $q\in S(\cB)$, and set $p_0=p{\upharpoonright}\cB^*_{\varphi}$ and $q_0=q{\upharpoonright}\cB_\varphi$. Then $dp^*(A)=d^*_{\varphi}p_0$ and $dq(A)=d_\varphi q_0$. By Theorem \ref{thm:deftypes}$(b)$, $d_\varphi q_0\in p_0$ if and only if $d^*_{\varphi}p_0\in q_0$. It follows that $d_\varphi q_0\in p$ if and only if $d^*_{\varphi}p_0\in q$, i.e., $dq(A)\in p$ if and only if $dp^*(A)\in q$. 
\end{proof}

We can now prove our first main result on stable subsets of groups.

\begin{theorem}\label{thm:left-inv}
Let $\cB\seq\cB^{\st}_G$ be a left-invariant Boolean algebra.
\begin{enumerate}[$(a)$]
\item $(S(\cB^\sharp),\ast)$ is a $G$-semigroup, and is isomorphic to the Ellis semigroup of $S(\cB)$.
\item $S(\cB)$ is a weakly almost periodic $G$-flow.
\item $\Sg(\cB)$ is the unique minimal subflow of $S(\cB)$.
\item $(\Sg(\cB^\sharp),\ast)$ is a profinite group.
\item There is a unique left-invariant probability measure  on $\cB$. 
\end{enumerate}
\end{theorem}
\begin{proof}
Part $(a)$ follows from Theorem \ref{thm:ESli} and Lemma \ref{lem:BAstab}$(a)$. 

Part $(b)$. Given $p\in S(\cB^\sharp)$ and $q\in S(\cB)$, let $f_p(q)=\{A\in\cB:dq(A)\in p\}$. By Theorem \ref{thm:ESli} and Lemma \ref{lem:BAstab}$(a)$, every element of $E(S(\cB))$ is of the form $f_p$ for some $p\in S(\cB^\sharp)$. Therefore, to show that $S(\cB)$ is weakly almost periodic, it suffices by Theorem \ref{thm:EN}$(a)$ to show that $f_p$ is continuous for all $p\in S(\cB^\sharp)$. So fix $p\in S(\cB^\sharp)$. We use the dual $dp^*$ notation from Lemma \ref{lem:BAstab}. Given $q\in S(\cB)$, it follows from Lemma \ref{lem:BAstab}$(c)$ that $f_p(q)=\{A\in\cB:dp^*(A)\in q\}$ (recall also that $dp^*(A)\in \cB$ for any $A\in\cB$ by Lemma \ref{lem:BAstab}$(b)$).  So for any $A\in\cB$, we have $f_p\inv(S_{\!A}(\cB))=S_{dp^*(A)}(\cB)$, which implies that $f_p$ is continuous. 

Part $(c)$. By parts $(a)$ and $(b)$ applied to $\cB^\sharp$, $S(\cB^\sharp)$ is weakly almost periodic and $(E(S(\cB^\sharp)),\circ)\cong (S(\cB^\sharp),\ast)$. So by Theorem \ref{thm:EN}$(b)(i)$ and Corollary \ref{cor:NewSB}, $\Sg(\cB^\sharp)$ is the unique minimal subflow of $S(\cB^\sharp)$. Now, if $p\in \Sg(\cB^\sharp)$ then $p{\upharpoonright}\cB\in \Sg(\cB)$. So $\Sg(\cB)$ is the unique minimal subflow of $S(\cB)$ by Corollary \ref{cor:NewSB}. 

Part $(d)$ follows from Theorem \ref{thm:EN}$(b)(iii)$ since $S(\cB^\sharp)\cong E(S(\cB^\sharp))$ is weakly almost periodic, with unique minimal subflow $\Sg(\cB^\sharp)$. 

Part $(e)$. By Theorem \ref{thm:EN}$(b)(iv)$, and parts $(b)$ and $(c)$, there is a unique $G$-invariant Borel probability measure on $S(\cB)$. So the claim  follows from the usual correspondence between regular Borel ($\sigma$-additive) probability measures on $S(\cB)$ and (finitely additive) probability measures on $\cB$ (see, e.g., \cite[Proposition 416Q]{Fremv4}). One only needs to check that this correspondence preserves $G$-invariance. 
\end{proof}

For later reference, we note the following consequence of Theorem \ref{thm:left-inv}$(c)$ and Corollary \ref{cor:NewSB}.

\begin{corollary}\label{cor:gentype}
Suppose $\cB\seq\cB^{\st}_G$ is a left-invariant Boolean algebra, and $\cF\seq\cB$ is closed under finite intersections and contains only generic sets. Then there is some $p\in \Sg(\cB)$ such that $\cF\seq p$.
\end{corollary}

\subsection{Measures and generic stable sets}

Note that if $\cB$ is a bi-invariant sub-algebra of $\cB^{\st}_G$, then $\cB^\sharp=\cB$ and so Theorem \ref{thm:left-inv} provides a full picture of the topological and algebraic  behavior of the $G$-flow $S(\cB)$. In particular, $(S(\cB),\ast)$ is a well-defined semigroup and $(\Sg(\cB),\ast)$ is a profinite group. In this subsection, we will use these results in the bi-invariant case to obtain some initial conclusions about the behavior of stable subsets of $G$. 

We start by noting that if $\cB$ is a bi-invariant sub-algebra of $\cB^{\st}_G$ then, since $(\Sg(\cB),\ast)$ is a profinite group, it admits a unique bi-invariant Borel probability measure (i.e., the normalized Haar measure). The next lemma makes an explicit connection between this measure and the measure given by Theorem \ref{thm:left-inv}$(e)$.

\begin{lemma}\label{lem:must}
Let $\cB$ be a bi-invariant sub-algebra of $\cB^{\st}_G$, and let $\mu$ be the unique left-invariant probability measure  on $\cB$ (which exists by Theorem \ref{thm:left-inv}(e)).  Then $\mu$ is bi-invariant and, for any $A\in\cB$, $\mu(A)$ is the normalized Haar measure of $\Sg_{\!A}(\cB)$.
\end{lemma}
\begin{proof}
Let $\eta$ denote the normalized Haar measure on $\Sg(\cB)$, and  define $\nu\colon\cB\to [0,1]$ such that $\nu(A)=\eta(\Sg_{\!A}(\cB))$. Then $\nu$ is a probability measure on $\cB$ (this uses the fact that if $A,B\in\cB$ then $\Sg_{\!A\cap B}(\cB)=\Sg_{\!A}(\cB)\cap \Sg_{\!B}(\cB)$ and $\Sg_{\!A\cup B}=\Sg_{\!A}(\cB)\cup\Sg_{\!B}(\cB)$). We will show that $\nu$ is bi-invariant. Given this, it will follow from Theorem \ref{thm:left-inv}$(e)$ that $\mu=\nu$, which gives us the desired results.  

Toward proving that $\nu$ is bi-invariant, fix $A\in \cB$ and $g\in G$. We want to show $\eta(\Sg_{\!gA}(\cB))=\eta(\Sg_{\!A}(\cB))=\eta(\Sg_{\!Ag}(\cB))$. Let $u$ be the identity in $\Sg(\cB)$. Given $p\in \Sg(\cB)$, we have 
$$
p\in \Sg_{\!gA}(\cB)\miff A\in g\inv p\miff A\in g\inv(u\ast p)\miff A\in g\inv u\ast p.
$$
Since $g\inv u\in \Sg(\cB)$, it follows that $\Sg_{\!gA}(\cB)=q\ast \Sg_{\!A}(\cB)$, where $q$ is the inverse of $g\inv u$ in $(\Sg(\cB),\ast)$. Similarly, $\Sg_{\!Ag}(\cB)=\Sg_{\!A}(\cB)\ast r$, where $r$ is the inverse of $ug\inv$. So we have the desired result by bi-invariance of $\eta$.
\end{proof}

Note that if $\mu$ is the unique bi-invariant probability measure on $\cB^{\st}_G$, and $\cB$ is a left-invariant sub-algebra of $\cB^{\st}_G$, then $\mu{\upharpoonright} \cB$ must be the unique left-invariant probability measure on $\cB$. So we can think of a stable set $A$ as having a uniquely defined measure $\mu(A$), which is independent of then ambient sub-algebra of $\cB^{\st}_G$.

\begin{theorem}\label{thm:gensets}
Let $\mu$ be the unique bi-invariant probability measure on $\cB^{\st}_G$.
\begin{enumerate}[$(a)$] 
\item If $A,B\in\cB^{\st}_G$ and $A\cup B$ is generic, then $A$ is generic or $B$ is generic.
\item If $A\in\cB^{\st}_G$ then $A$ is generic or $G\backslash A$ is generic.
\item If $A\in\cB^{\st}_G$ then $\mu(A\inv)=\mu(A)$.
\item Given $A\in\cB^{\st}_G$, the following are equivalent:
\begin{enumerate}[$(i)$]
\item $A$ is generic (i.e, $G=FA$ for some finite $F\seq G$);
\item $G=AF$ for some finite $F\seq G$;
\item $G=EAF$ for some finite $E,F\seq G$;
\item $\mu(A)>0$.
\end{enumerate}
\end{enumerate}
\end{theorem}
\begin{proof}
Part $(a)$. If $A\cup B$ is generic then, by Corollary \ref{cor:gentype} there is some generic type $p\in S(\cB^{\st}_G)$ such that $A\cup B\in p$. Now $A\in p$ or $B\in p$ since $p$ is a type.

Part $(b)$ is immediate from part $(a)$. 

Part $(c)$. By Lemma \ref{lem:must}, $A\mapsto \mu(A\inv)$ is a left-invariant probability measure on $\cB^{\st}_G$. So the claim follows from Theorem \ref{thm:left-inv}$(e)$.

Part $(d)$. $(i)\Rightarrow (iv)$ follows from left invariance and finite additivity of $\mu$.

$(iv)\Rightarrow (i)$. By Lemma \ref{lem:must}, $\mu(A)$ is the normalized Haar measure of $\Sg_{\!A}(\cB^{\st}_G)$. So if $\mu(A)>0$ then $\Sg_{\!A}(\cB^{\st}_G)$ is nonempty, and thus $A$ is generic.

$(i)\Rightarrow (ii)$. Suppose $A$ is generic. Then $A\inv$ is generic by $(i)\Leftrightarrow (iv)$ and part $(c)$. Thus $G=FA\inv$ for some finite $F\seq G$, and so $G=AF\inv$.  

$(ii)\Rightarrow (iii)$ is trivial. 

$(iii)\Rightarrow (iv)$.  If $G=EAF$ for some finite $E,F\seq G$ then, by finite additivity, there are $g,h\in G$ such that $\mu(gAh)>0$. So $\mu(A)>0$ by Lemma \ref{lem:must}.
\end{proof}

\begin{remark}
In part $(d)$ of the previous theorem, the equivalence of $(i)$ and $(ii)$ is trivial if $G$ is abelian. So we note that it is nontrivial in general. For example,  let $G$ be the free group on two generators $a$ and $b$, and let $A$ be the set of words in $G$ that start with $a$. Then $A$ is (left) generic since $G=ba\inv A\cup a\inv A$, but not right generic (i.e., fails condition $(ii)$). 
\end{remark}

\subsection{Algebraic structure of $\Sg(\cB)$}

Let $\cB$ be a left-invariant sub-algebra of $\cB^{\st}_G$. We have seen that  if $\cB$ is bi-invariant then $(\Sg(\cB),\ast)$ is a profinite group. The goal of this subsection is to show that even without bi-invariance, $\Sg(\cB)$ still exhibits algebraic structure compatible with the topology. We first set some terminology. 

 \begin{definition}
Let $C$ be compact Hausdorff group. A \textbf{homogeneous $C$-space} is a Hausdorff space $X$ together with a transitive continuous group action $C\times X\to X$.
\end{definition}

Given an arbitrary group $H$ and an arbitrary subgroup $K\leq H$, we let $H/K$ denote the set of left cosets of $K$ in $H$.
By the general theory, any homogeneous $C$-space can be identified with $C/K$ for some closed subgroup $K$ (see, e.g., \cite[Proposition 1.10]{HofMo3}). We will show that if $\cB$ is as above, then $\Sg(\cB)$ is a profinite homogeneous $\Sg(\cB^\sharp)$-space. This will be a specific instance of the following general fact, which is largely due to Auslander \cite{Ausbook}.

\begin{lemma}\label{lem:HSgen}
Let $S$ be a weakly almost periodic $G$-flow, and $X\seq S$ a minimal subflow. Let $C$ be the unique minimal subflow of $E(S)$, and recall that $(C,\circ)$ is a compact Hausdorff group (see Theorem \ref{thm:EN}). Then $X$ is a homogeneous $C$-space, witnessed by the action $(f,p)\mapsto f(p)$. Therefore, if $x_0\in X$ is a fixed point and $K=\{f\in C:f(x_0)=x_0\}$ is the stabilizer of $x_0$, then $K$ is closed subgroup of $C$, and $C/K$ and $X$ are isomorphic as homogeneous $C$-spaces via $f\circ K\mapsto f(x_0)$.
\end{lemma}
\begin{proof}
The $G$-flow $X$ is ``almost periodic", and so this statement elaborates slightly on \cite[Theorem 3.6]{Ausbook}. We include details for the sake of clarity. First, note that since $X$ is closed and $G$-invariant, it is closed under any $f\in E(S)$ by definition of $E(S)$. For the rest of the proof, we restrict to $p\in X$ and $f\in C$.

To verify that $(f,p)\mapsto f(p)$ is a group action, we just need to check that the identity $u\in C$ is the identity map on $X$. We follow the proof of \cite[Proposition II.8(1)]{EllNer}. Recall that for $a\in G$, $\pi_a$ denotes the map $p\mapsto ap$, which is in $E(S)$. Now fix some $p\in X$. Then, for any $a\in G$, we have 
$$
u(au(p))=(u\circ \pi_a\circ u)(p)=(\pi_a\circ u\circ u)(p)=(\pi_a\circ u)(p)=au(p),
$$ 
where the second and third equalities use Theorem \ref{thm:EN}$(b)(ii)$. So $u$ is the identity on the $G$-orbit of $u(p)$, which is dense in $X$ since $X$ is minimal. Since $u$ is continuous (by Theorem \ref{thm:EN}$(a)$), it follows that $u$ is the identity on $X$. 

Next, note that for a fixed $f\in C$, $p\mapsto f(p)$ is continuous by Theorem \ref{thm:EN}$(a)$; and for a fixed $p\in C$, $f\mapsto f(p)$ is continuous by definition of the topology on $C$. So the action $(f,p)\mapsto f(p)$ is separately continuous and thus continuous by a result of Ellis  \cite{EllisLCTG}. (However, separated continuity will suffice for the remaining claims here, and thus for all of our main results.)

We now prove transitivity. Fix $p\in X$, and set $Y=\{f(p):f\in C\}$. We show that $X=Y$. Since $f\mapsto f(p)$ is continuous, and $C$ is compact, it follows that $Y$ is closed in $X$. Since $C$ is a $G$-flow via $(a,f)\mapsto \pi_a\circ f$, it also follows that $Y$ is $G$-invariant. So $X=Y$ since $X$ is minimal. 

The rest now follows from basic facts about group actions, and only requires continuity of $f\mapsto f(x_0)$ (see, e.g., \cite[Proposition 1.10]{HofMo3}).
\end{proof}

We now apply Lemma \ref{lem:HSgen} to the $G$-flow $S(\cB)$, where $\cB$ is a left-invariant sub-algebra of $\cB^{\st}_G$. In order to obtain a more explicit statement, we will make a special choice for the fixed point $x_0$ referenced in the previous lemma. 

\begin{corollary}\label{cor:HSstone}
Suppose $\cB$ is a left-invariant sub-algebra of $\cB^{\st}_G$. Then $\Sg(\cB)$ is a profinite homogeneous $\Sg(\cB^\sharp)$-space. In particular,  let $u$ be the identity in $(\Sg(\cB^\sharp),\ast)$, and set $K=\{p\in \Sg(\cB^\sharp):p{\upharpoonright}\cB=u{\upharpoonright}{\cB}\}$. Then $K$ is closed subgroup of $\Sg(\cB^\sharp)$, and $\Sg(\cB^\sharp)/K$ is isomorphic to $\Sg(\cB)$, as a homogeneous $\Sg(\cB^\sharp)$-space, via the map $p\ast K\mapsto p{\upharpoonright}\cB$.
\end{corollary}
\begin{proof}
By Theorem \ref{thm:left-inv} and Lemma \ref{lem:HSgen}, $\Sg(\cB)$ is a profinite homogeneous $\Sg(\cB^\sharp)$-space. This goes through the $G$-semigroup isomorphism $\Phi\colon S(\cB^\sharp)\to E(S(\cB))$ given by Theorem \ref{thm:ESli} (via Lemma \ref{lem:BAstab}$(a)$). In particular, let $u$ be the identity in $\Sg(\cB^\sharp)$, and set $u_0=u{\upharpoonright}\cB$ and $K'=\{p\in S(\cB^\sharp):\Phi(p)(u_0)=u_0\}$. Then $K'$ is a closed subgroup of $\Sg(\cB^\sharp)$, and $\Sg(\cB^\sharp)/K'$ is isomorphic to $\Sg(\cB)$, as a homogeneous $\Sg(\cB^\sharp)$-space, via the map $p\ast K'\mapsto \Phi(p)(u_0)$. Therefore, to prove the claim, we only need to verify that if $p\in \Sg(\cB^\sharp)$ then $\Phi(p)(u_0)=p{\upharpoonright}\cB$. So fix $p\in \Sg(\cB^\sharp)$ and $A\in\cB$. Then $du_0(A)=du(A)$ by definition of $u_0$. By definition of $\Phi$, we have
$$
A\in \Phi(p)(u_0)\miff du_0(A)\in p\miff du(A)\in p\miff A\in p\ast u\miff A\in p.
$$
Therefore $\Phi(p)(u_0)=p{\upharpoonright}\cB$. 
\end{proof}

\subsection{The structure of stable sets}
We are now ready to start discussing subgroups of $G$. Let $\cB\seq\cP(G)$ be a left-invariant Boolean algebra, and suppose $H$ is a finite-index subgroup of $G$, which is in $\cB$.
Then $G$ is partitioned into finitely many left cosets of $H$, each of which is in $\cB$. Therefore, any type $p\in S(\cB)$ must contain a unique left coset of $H$. This motivates the following notation.

\begin{definition}
Let $\cB\seq\cP(G)$ be a left-invariant Boolean algebra. We write $H\leq^{\cB}_{\fin} G$ to denote that $H$ is a finite-index subgroup of $G$ and $H\in\cB$. Given $p\in S(\cB)$ and $H\leq^{\cB}_{\fin}G$, we let $pH$ denote the unique left coset of $H$ in $p$.
\end{definition}

For example, if $p=p^{\cB}_a$ is the principal $\cB$-type on $a\in G$, then $pH=aH$.

\begin{lemma}\label{lem:cosets}
Let $\cB$ be a bi-invariant sub-algebra of $\cB^{\st}_G$. Fix $p,q\in S(\cB)$ and $H\leq^{\cB}_{\fin} G$. Then $(p\ast q)H=pH$ if and only if $qH=H$. Moreover, if $H$ is normal, then $(p\ast q)H=pH\cdot qH$ (where $\cdot$ denotes the group operation in $G/H$).
\end{lemma}
\begin{proof}
 Fix $a,b\in G$ such that $pH=aH$ and $qH=bH$. Note that $dq(aH)=aHb\inv$. So if $aH\in p\ast q$, then $aHb\inv\in p$, which implies $aHb\inv\cap aH\neq\emptyset$, and thus $b\in H$. Conversely, if $b\in H$ then $dq(aH)=aH\in p$, and so $aH\in p\ast q$. Finally, if $H$ is normal then $dq(abH)=abHb\inv=aH\in p$, and so $abH\in p\ast q$. 
\end{proof}

We now state and prove the main technical result that will allow us to approximate stable sets using finite-index subgroups of $G$.

\begin{proposition}\label{prop:LItop}
Let $\cB$ be a left-invariant sub-algebra of $\cB^{\st}_G$. Then 
\[
\left\{\Sg_{\!aH}(\cB):a\in G,~H\leq^{\cB}_{\fin} G\right\}
\]
is a basis for the topology on $\Sg(\cB)$.
\end{proposition}
\begin{proof}
It suffices to show that for any nonempty open set $U\seq \Sg(\cB)$ and $r\in U$, there is a subgroup $H\leq^{\cB}_{\fin} G$ such that $\Sg_{rH}(\cB)\seq U$. To do this, we will exploit the description of $\Sg(\cB)$ as a profinite homogeneous $\Sg(\cB^\sharp)$-space. So let  $u$ be the identity in $(\Sg(\cB^\sharp), \ast)$ and let $K=\{p\in\Sg(\cB^\sharp):p{\upharpoonright}\cB=u{\upharpoonright}\cB\}$. Then $K$ is a closed subgroup of $\Sg(\cB^\sharp)$, and $\Sg(\cB^\sharp)/K$ is isomorphic to $\Sg(\cB)$, as a homogeneous space, via the map $\pi\colon p\ast K\mapsto p{\upharpoonright}\cB$ (see Corollary \ref{cor:HSstone}). Now, if $A\in\cB$ then $\Sg_{\!A}(\cB^\sharp)$ is $K$-invariant and, moreover, $\pi(\Sg_{\!A}(\cB^\sharp)/K)=\Sg_{\!A}(\cB)$. So this allows us to translate the main goal to $\Sg(\cB^\sharp)/K$. In particular, we want to show that for any nonempty
open set $U\seq \Sg(\cB^\sharp)/K$, and  $r\in\Sg(\cB^\sharp)$ with $r\ast K\in U$, there is some $H\leq^{\cB}_{\fin} G$ such that $\Sg_{rH}/K\seq U$. So fix such $U$ and $r$.

Let $V=\{p\in\Sg(\cB^\sharp):p\ast K\in U\}$ be the pullback of $U$ to $\Sg(\cB^\sharp)$. Then $V$ is open in $\Sg(\cB^\sharp)$, and $r\ast K\seq V$, i.e., $K\seq r\inv\ast V$. Since $\Sg(\cB^\sharp)$ is a profinite group and $K$ is a closed subgroup, it follows that $K$ is the intersection of all open subgroups containing $K$ (see \cite[Proposition 2.1.4]{RZbook}). Since $r\inv \ast V$ is open, and any open subgroup  is also closed, it follows that there is an open subgroup $L$ of $\Sg(\cB^\sharp)$ such that $K\leq L\seq r\inv \ast V$. Note that $L$ has finite index in $\Sg(\cB^\sharp)$ since $\Sg(\cB^\sharp)$ is compact. Altogether, $L$ is a finite-index clopen subgroup of $\Sg(\cB^\sharp)$. 

We will use $L$ to obtain our desired subgroup $H\leq^{\cB}_{\fin} G$. To do this, we need to represent $\Sg(\cB^\sharp)$ as a compactification of $G$. In particular, define $\sigma\colon G\to \Sg(\cB^\sharp)$ such that $\sigma(g)=gu$. Then $\sigma(G)$ is dense in $\Sg(\cB^\sharp)$ since $\Sg(\cB^\sharp)$ is a minimal flow.  Given $g,h\in G$, if we let $p=p^{\cB^\sharp}_g$ be the principal $\cB^\sharp$-type on $g$, then 
\[
\sigma(gh)=ghu=p\ast hu= p\ast u\ast hu=gu\ast hu=\sigma(g)\ast\sigma(h).
\]
So $\sigma$ is a homomorphism.

Since $L/K$ is clopen in $\Sg(\cB^\sharp)/K$, it follows that $L/K=\pi\inv(\Sg_{\!A}(\cB))$ for some $A\in\cB$. Let $H=\sigma\inv(L)$. Then $H$ is a finite-index subgroup of $G$. We show that $H=\{g\in G:A\in gu\}$. First, if $g\in H$, then $gu\ast K=p\ast K$ for some $p\in \Sg_{\!A}(\cB^\sharp)$, and so $A\in gu$ since $gu{\upharpoonright}\cB=p{\upharpoonright}\cB$. Conversely, suppose $A\in gu$. Then $gu\in \Sg_{\!A}(\cB^\sharp)$, and so $gu\ast K\in L/K$. Therefore $gu\ast K=p\ast K$ for some $p\in L$, and thus $p\inv\ast gu\in K\leq L$. So $gu\in L$, i.e., $g\in H$. 

By Theorem \ref{thm:EN}$(b)(ii)$, $gu=ug$ for any $g\in G$. Using the notation of Lemma \ref{lem:BAstab}, we now have $H=\{g\in G:Ag\inv\in u\}=du^*(A)$. So $H\in\cB$ by Lemma \ref{lem:BAstab}$(b)$. 

Next, we show that $\Sg_{\!H}(\cB^\sharp)\seq L$. Suppose $\Sg_{\!H}(\cB^\sharp)\backslash L$ is nonempty. Since $\Sg_{\!H}(\cB^\sharp)\backslash L$ is open, and $\sigma(G)$ is dense in $\Sg(\cB^\sharp)$, there is some $g\in G$ such that $gu\in \Sg_{\!H}(\cB^\sharp)\backslash L$. Since $gu\in \Sg_{\!H}(\cB^\sharp)$, we have $H\in gu$, and so $gH=H$ by Lemma \ref{lem:cosets}. Since $H=\sigma\inv(L)$, we have $gu\in L$, which is a contradiction. 

We now have $\Sg_{\!H}(\cB^\sharp)\seq L\seq r\inv\ast V$. By Lemma \ref{lem:cosets}, we also have $K\seq \Sg_{\!H}(\cB^\sharp)$ and $r\ast \Sg_{\!H}(\cB^\sharp)=\Sg_{rH}(\cB^\sharp)$. So $\Sg_{rH}(\cB^\sharp)\seq V$, and thus $\Sg_{rH}(\cB^\sharp)/K\seq U$. 
\end{proof}

We again let $\mu$ denote the unique bi-invariant probability measure on $\cB^{\st}_G$. The following is our main result on the approximate structure of stable subsets of $G$.

\begin{corollary}\label{cor:structure}
Let $\cB$ be a left-invariant sub-algebra of $\cB^{\st}_G$, and fix $A\in\cB$. Then there is a finite-index subgroup $H\leq G$, which is in $\cB$, and a set $Y\seq G$, which is a union of left cosets of $H$, such that $\mu(A\smd Y)=0$. Consequently,  if $\cC$ is the set of left cosets of $H$ contained in $Y$, then $\mu(A)=\mu(Y)=|\cC|/[G:H]$. 
\end{corollary}
\begin{proof}
By Proposition \ref{prop:LItop}, there are $H_1,\ldots,H_n\leq^{\cB}_{\fin} G$ and $g_1,\ldots,g_n\in G$ such that 
\[
\Sg_{\!A}(\cB)=\Sg_{\!g_1H_1}(\cB)\cup\ldots\cup \Sg_{\!g_nH_n}(\cB)=\Sg_{\!g_1H_1\cup\ldots\cup  g_nH_n}(\cB).
\]
Let $H=H_1\cap\ldots\cap H_n$ and let $Y=g_1H_1\cup\ldots\cup g_nH_n$. Then $H\leq^{\cB}_{\fin} G$, $Y$ is a union of left cosets of $H$, and $\Sg_{\!A}(\cB)=\Sg_{Y}(\cB)$. If $A\smd Y$ is generic then, by Corollary \ref{cor:gentype},  there is some $p\in\Sg(\cB)$ such that $A\smd Y\in p$, which  contradicts $\Sg_{\!A}(\cB)=\Sg_{Y}(\cB)$. So $A\smd Y$ is not generic, and thus $\mu(A\smd Y)=0$ by Theorem \ref{thm:gensets}$(d)$. The remaining claims follow from  invariance and finite additivity of $\mu$. 
\end{proof}

\subsection{A note on $k$-stable sets}
We say that $A\seq G$ is \textbf{$k$-stable} if the relation $xy\in A$ on $G\times G$ is $k$-stable (see Definition \ref{def:RBA}). Using similar methods as in Fact \ref{fact:BSTG}, one can show that the collection of subsets of $G$ that are $k$-stable for some $k\geq 1$ forms a bi-invariant sub-algebra of $\cB^{\st}_G$, which includes all subgroups of $G$. (In fact, a nonempty subset of a group is $2$-stable if and only if it is a coset of a subgroup.) The next example demonstrates that the $k$-stable sets may form a proper sub-algebra.

\begin{example}
Let $G$ be the group of integers $(\Z,+)$. For each $n\geq 1$, choose an $n$-term arithmetic progression $A_n\seq\mathbb{Z}$, with common difference $n$, so that $\min A_n>\max A_{n-1}+n^2$. Let $A=\bigcup_{n=1}^\infty A_n$. Then $A$ is stable but not $k$-stable for any $k\geq 1$. (We leave this as an exercise for the reader.)
\end{example}
On the other hand, stable sets are ``close" to $k$-stable sets.

\begin{corollary}\label{cor:stabCR}
For any stable $A\seq G$ there is $Y\seq G$ such that $\mu(A\smd Y)=0$ and $Y$ is $k$-stable for some $k\geq 1$.
\end{corollary}
\begin{proof}
By Corollary \ref{cor:structure}, we have $Y\seq G$, which is a union of cosets of a fixed finite-index subgroup $H\leq G$, such that $\mu(A\smd Y)=0$. It is easy to show that $Y$ is $k$-stable for some $k\leq [G:H]+1$ (see, e.g., \cite[Lemma 1.5]{SandersSR}).
\end{proof}

\subsection{Profinite completions}
In this subsection, we explicitly describe $\Sg(\cB)$ as a projective limit of finite coset spaces, where $\cB\seq\cB^{\st}_G$ is a left-invariant sub-algebra.

\begin{definition}\label{def:projlim}
Given a Boolean algebra $\cB\seq\cP(G)$,  define $\hat{G}_{\cB}=\varprojlim_{H\leq^{\cB}_{\fin} G}G/H$.
\end{definition}

We first make some observations on the bi-invariant case. Suppose $\cB\seq\cP(G)$ is a bi-invariant Boolean algebra, and let $\cN$ be the collection of all finite-index \emph{normal} subgroups of $G$ in $\cB$. Then $\cN$ is co-initial in the the collection of all $H\leq^{\cB}_{\fin} G$, and so $\hat{G}_{\cB}$ is a profinite \emph{group} isomorphic to $\varprojlim_{N\in\cN}G/N$ (see, e.g., \cite[Lemma 1.19]{RZbook}). For example, note that if $\cB$ contains all finite-index normal subgroups of $G$, then $\hat{G}_{\cB}$ is the classical profinite completion of $G$. See \cite[Section 3.2]{RZbook} for details.

\begin{corollary}\label{cor:projlim}
Let $\cB$ be a left-invariant sub-algebra of $\cB^{\st}_G$.
\begin{enumerate}[$(a)$]
\item  If  $\cB$ is bi-invariant then $\Sg(\cB)$ and $\hat{G}_{\cB}$ are isomorphic profinite groups.
\item $\Sg(\cB)$ and $\hat{G}_{\cB}$ are isomorphic homogeneous $\hat{G}_{\cB^\sharp}$-spaces. 
\end{enumerate}
\end{corollary} 
\begin{proof}
We first show that $\Sg(\cB)$ and $\hat{G}_{\cB}$ are homeomorphic. Let $\cH$ denote the collection of all $H\leq^{\cB}_{\fin} G$. Define $\tau\colon \Sg(\cB)\to\hat{G}_{\cB}$ such that $\tau(p)=(pH)_{H\in\cH}$. It is straightforward to check that $\tau$ is continuous. Moreover, $\tau$ is injective by Proposition \ref{prop:LItop} and surjective by Corollary \ref{cor:gentype}.

Now suppose for a moment that $\cB$ is bi-invariant, and let $\cN$ be the collection of all \emph{normal} $H\in\cH$. As indicated above, any element of $\hat{G}_{\cB}$ is completely determined by its restriction to $\cN$, and so we can identify $\hat{G}_{\cB}$ with the profinite group $\varprojlim_{N\in\cN} G/N$. In this case, $\tau$ is a group isomorphism by Lemma \ref{lem:cosets}. 

Finally, we return to the case that $\cB$ is only left-invariant. By the previous arguments and Corollary \ref{cor:HSstone},  $\Sg(\cB)$ is a homogeneous $\hat{G}_{\cB^\sharp}$-space. Now redefine $\cN$ to be the collection of all finite-index normal subgroups of $G$ in $\cB^\sharp$. Given $N\in\cN$, let $H_N=\bigcap\{H\in\cH:N\leq H\}$, and note that $H_N\in\cH$. For any $H\in\cH$, if $N=\bigcap_{g\in G}gHg\inv$, then $N\in\cN$ and $H_N\leq H$. So any element of $\hat{G}_{\cB}$ is completely determined by its restriction to $\{H_N:N\in\cN\}$. Altogether, $\hat{G}_{\cB}$ can be identified with $\varprojlim_{\cN} G/H_N$, which is a homogeneous $\hat{G}_{\cB^\sharp}$-space isomorphic to  $\hat{G}_{\cB^\sharp}/K$, where $K=\varprojlim_{\cN} H_N/N$. Finally, one checks that the homeomorphism $\tau\colon \Sg(\cB)\to \hat{G}_{\cB^\sharp}$ is an isomorphism of homogeneous $\hat{G}_{\cB^\sharp}$-spaces.  
\end{proof}

Next we make some remarks motivated by certain notions from the model theory of groups. This will also be in preparation for the results in Section \ref{sec:CC}.  

\begin{definition}\label{def:G0}
Given a Boolean algebra $\cB\seq\cP(G)$, define $G^0_{\!\cB}$ to be the intersection of all finite-index subgroups of $G$ that are in $\cB$.
\end{definition}

Note that if $\cB\seq\cP(G)$ is a bi-invariant Boolean algebra, then $G^0_{\!\cB}$ coincides with the intersection of all finite-index \emph{normal} subgroups of $G$ that are in $\cB$ (in particular, $G^0_{\!\cB}$ is a normal subgroup of $G$). 

\begin{remark}\label{rem:global}
Suppose $\cB$ is a bi-invariant sub-algebra of $\cP(G)$, and let $\cN$ denote the collection of all finite-index normal subgroups of $G$ in $\cB$. Then we have a canonical homomorphism $\rho\colon G\to\hat{G}_{\cB}$ such that $\rho(a)=(aN)_{N\in\cN}$. Note that $\ker\rho=G^0_{\!\cB}$, and so $G/G^0_{\!\cB}$ embeds as a dense subgroup of $\hat{G}_{\cB}$ via the induced quotient map $\rho^*$. The coarsest topology on $G$ that makes $\rho$ continuous is the \emph{$\cB$-profinite topology}, whose basic open sets are cosets of $N\in\cN$. Note that $G^0_{\!\cB}$ is closed in this topology and, moreover, the topology on $G/G^0_{\!\cB}$ induced from $\hat{G}_{\cB}$ by $\rho^*$  coincides with the quotient of the $\cB$-profinite topology. In fact, $\rho^*\colon G/G^0_{\!\cB}\to \hat{G}_{\cB}$ is a homeomorphism  if and only if the $\cB$-profinite topology on $G$ is compact. Now suppose $\cB\seq\cB^{\st}_G$. Then $\rho=\tau\sigma$, where $\sigma\colon G\to\Sg(\cB)$ is as in the proof of Proposition \ref{prop:LItop} and $\tau$ is as in the proof of Corollary \ref{cor:projlim}. So $G/G^0_{\!\cB}$ embeds as a dense subgroup of $\Sg(\cB)$ via $\sigma^*:=\tau\inv\rho^*$. Once again, $\sigma^*\colon G/G^0_{\!\cB}\to \Sg(\cB)$ is a homeomorphism if and only if the $\cB$-profinite topology on $G$ is compact. 

An example from model theory  is when $G$ is definable (say, over $\emptyset$) in some model $M$ of a stable theory $T$, and $\cB$ is the Boolean algebra of definable subsets of $G$. Then $G^0_{\cB}$ is an intersection of at most $|T|$ groups in $\cN$ (see \cite[Chapter 5]{PoStG}; this also follows from Theorem \ref{thm:sif} below). Now suppose $M$ is $|T|^+$-saturated. Then the $\cB$-profinite topology on $G$ has a basis of cardinality at most $|T|$, and hence is compact.  So $G/G^0_{\!\cB}$ is homeomorphic to $\Sg(\cB)$ and $\hat{G}_{\cB}$.
\end{remark}

\begin{definition}
Let $\cB\seq\cP(G)$ be a left-invariant Boolean algebra. Given $p\in S(\cB)$, define $\Stab(p)=\{g\in G:gp=p\}$. 
\end{definition}

\begin{corollary}\label{cor:stabGl}
Suppose $\cB$ is a left-invariant sub-algebra of $\cB^{\st}_G$, and $p\in\Sg(\cB)$. Given $H\leq^{\cB}_{\fin} G$, let $H^p=aHa\inv$, where $aH=pH$. Then $\Stab(p)=\bigcap_{H\leq^{\cB}_{\fin}G}H^p$. Moreover, if $\cB$ is bi-invariant then $\Stab(p)=G^0_{\!\cB}$. 
\end{corollary}
\begin{proof}
Note that the second claim follows from the first by the observation made after Definition \ref{def:G0}. For the first claim, fix $p\in\Sg(\cB)$ and $g\in G$. By Proposition \ref{prop:LItop}, $g\in\Stab(p)$ if and only if $(gp)H=pH$ for all $H\leq^{\cB}_{\fin} G$. Also, if $H\leq^{\cB}_{\fin} G$ then $(gp)H=pH$ if and only if $g\in H^p$.
\end{proof}

\section{Connected components}\label{sec:CC}

Let $G$ be a group. As usual, we let $\mu$ denote the unique bi-invariant probability measure on $\cB^{\st}_G$. In this section, we analyze left-invariant sub-algebras  of $\cB^{\st}_G$ that contain a \emph{smallest} finite-index subgroup of $G$.
Recall that for a Boolean algebra  $\cB\seq\cP(G)$, $G^0_{\!\cB}$ denotes intersection of all finite-index subgroups of $G$ in $\cB$. 

\begin{lemma}\label{lem:VCeqs}
Let $\cB\seq\cP(G)$ be a Boolean algebra. The following are equivalent.
\begin{enumerate}[$(i)$]
\item $\cB$ contains only finitely many finite-index subgroups of $G$.
\item $\cB$ contains a smallest finite-index subgroup of $G$.
\item $G^0_{\!\cB}$ is the smallest finite-index subgroup of $G$ in $\cB$.
\item $G^0_{\!\cB}$ has finite index in $G$.
\end{enumerate}
\end{lemma}
\begin{proof}
The following implications are trivial: $(i)\Rightarrow(ii)\Rightarrow(iii)\Rightarrow(iv)$. For $(iv)\Rightarrow (i)$, note that any subgroup $H\leq^{\cB}_{\fin} G$ is a union of cosets of $G^0_{\!\cB}$. 
\end{proof}

As previously stated, the focus of this section will be on Boolean algebras satisfying the equivalent conditions of the previous lemma. So, for the sake of brevity, we will usually invoke these conditions using the formulation in $(iv)$.

\begin{definition}
Let $\cB$ be a left-invariant sub-algebra of $\cB^{\st}_G$. Given $X\in\cB$, define $\Stab_\mu(X)=\{g\in G:\mu(gX\smd X)=0\}$.
\end{definition}

\begin{remark}\label{rem:StabmuX}
Note that in the context of the previous definition, $\Stab_\mu(X)$ is a subgroup of $G$. Moreover, by Theorem \ref{thm:gensets}$(d)$ and Corollary \ref{cor:gentype}, we have
\[
\Stab_\mu(X)=\{g\in G:gX\smd X\text{ is not generic}\}=\bigcap_{p\in\Sg(\cB)}\{g\in G:gX\smd X\not\in p\}.
\]
\end{remark}

The following is the main result of this section.

\begin{theorem}\label{thm:local}
Let $\cB$ be a left-invariant sub-algebra of $\cB^{\st}_G$, and suppose $G^0_{\!\cB}$ has finite index in $G$.
\begin{enumerate}[$(1)$]

\item \textnormal{(Generic $\cB$-types and their stabilizers)}
\begin{enumerate}[$(a)$]
\item The map $p\mapsto pG^0_{\!\cB}$ is an action-preserving bijection from $S^g(\cB)$ to $G/G^0_{\!\cB}$. 
\item If $p\in \Sg(\cB)$  then $\Stab(p)=aG^0_{\!\cB}a\inv$ for some/any $a\in pG^0_{\!\cB}$.
\end{enumerate}
\vspace{3pt}

\item \textnormal{(Structure for sets in $\cB$)} If $X\in\cB$ then, for any left coset $C$ of $G^0_{\!\cB}$, either $\mu(C\cap X)=0$ or $\mu(C\backslash X)=0$. Moreover, if $\cC$ is the set of cosets $C$ of $G^0_{\!\cB}$ such that $\mu(C\cap X)>0$, and $Y=\bigcup\cC$, then $\mu(X)=\mu(Y)=|\cC|/[G:G^0_{\!\cB}]$.
\vspace{3pt}

\item \textnormal{(Stabilizers of sets in $\cB$)} If $X\in\cB$ then $\Stab_\mu(X)$ is a finite-index subgroup of $G$ in $\cB^\sharp$. Moreover,
\[
\textstyle G^0_{\!\cB^\sharp}=\bigcap_{a\in G}aG^0_{\!\cB} a\inv=\bigcap_{p\in\Sg(\cB)}\Stab(p)=\bigcap_{X\in\cB}\Stab_\mu(X),
\]
and so $G^0_{\!\cB^\sharp}$ is a finite-index normal subgroup of $G$ in $\cB^\sharp$.
\vspace{3pt}

\item \textnormal{(The bi-invariant case)} Assume $\cB$ is bi-invariant. Then $G^0_{\!\cB}$ is normal and the map in $(1a)$ is a group isomorphism. Moreover, if $p\in \Sg(\cB)$ then $\Stab(p)=G^0_{\!\cB}$; and if $X\in\cB$ then $\Stab_\mu(X)$ is a finite-index subgroup of $G$ in $\cB$. 
\end{enumerate}
\end{theorem}
\begin{proof}
Part $(1)$. Claims $(a)$ and $(b)$ follow from Corollaries \ref{cor:projlim} and \ref{cor:stabGl}.
\vspace{3pt}

Part $(2)$. For the first statement, it suffices to show that for any $X\in\cB$ and any left coset $C$ of $G^0_{\!\cB}$, exactly one of $C\cap X$ or $C\backslash X$ is generic. For this, let $p$ be the unique type in $\Sg(\cB)$ containing $C$ (which exists by part $(1)$).  By uniqueness and Corollary \ref{cor:gentype}, if $Y\in\cB$ and $Y\seq C$, then $Y$ is generic if and only if $Y\in p$. Moreover, since $C\in p$, we have that exactly one of $C\cap X$ or $C\backslash X$ is in $p$, as desired. 
 
For the second statement, it follows from Corollary \ref{cor:structure}, that there is a set $Y$, which is a union of cosets of $G^0_{\!\cB}$ such that $\mu(X\smd Y)=0$. So by finite additivity, a coset $C$ of $G^0_{\!\cB}$ is contained in $Y$ if and only if $\mu(C\cap X)>0$. 
\vspace{3pt}

Part $(3)$. We first prove the chain of equalities for $G^0_{\cB^\sharp}$. By part (1) and the definition of $G^0_{\!\cB^\sharp}$, we have $G^0_{\!\cB^\sharp}\seq \bigcap_{a\in G}aG^0_{\!\cB}a\inv=\bigcap_{p\in \Sg(\cB)}\Stab(p)$ . So it suffices to show
\[
\bigcap_{p\in\Sg(\cB)}\Stab(p)\seq \bigcap_{X\in\cB}\Stab_\mu(X)\seq G^0_{\!\cB^\sharp}.
\]
For the first containment, fix $a\in G$, and suppose $g\not\in\Stab_\mu(X)$ for some $X\in\cB$. Then $gX\smd X$ is generic (recall Remark \ref{rem:StabmuX}), and so there is some $p\in\Sg(\cB)$ containing $gX\smd X$ by Corollary \ref{cor:gentype}. It follows that $a\not\in\Stab(p)$. 
For the second containment, fix $g\in\bigcap_{X\in\cB}\Stab_\mu(X)$. We claim that $g\in \Stab_\mu(Y)$ for any $Y\in\cB^\sharp$. To see this, it suffices by definition of $\cB^\sharp$ to assume $Y=Xh$ for some $X\in\cB$ and $h\in G$. In this case, we have $\mu(gX\smd X)$ by assumption, and so $\mu(gY\smd Y)=0$ by right-invariance of $\mu$. So $g\in\Stab_\mu(Y)$. Now, notice that $G^0_{\!\cB^\sharp}=\bigcap_{H\leq^{\cB^\sharp}_{\fin}G}H=\bigcap_{H\leq^{\cB^\sharp}_{\fin}G}\Stab_\mu(H)$, which yields the desired result. 

Finally, fix $X\in\cB^\sharp$. Then $\Stab_\mu(X)$ has finite index in $G$ since it contains $G^0_{\!\cB^\sharp}$ by the above. So it remains to show that $\Stab_\mu(X)$ is in $\cB^\sharp$. Since $\Sg(\cB)$ is finite (by part (1)), it suffices by Remark \ref{rem:StabmuX} to fix $p\in \Sg(\cB)$ and show that the set $A=\{g\in G:g\inv X\smd X\in p\}$ is in $\cB^\sharp$. Let $\varphi(x,y)$ be the relation $x\in y\inv X\smd X$ on $G\times G$. Note that $\varphi(x,y)$ is stable and $\cB_\varphi\seq\cB^\sharp$. Let $p_0=p{\upharpoonright}\cB_\varphi$. Then $A=d_\varphi p_0\in\cB^*_{\varphi}$ by Theorem \ref{thm:deftypes}$(a)$. So it suffices to show that $\varphi(a,G)\in\cB^\sharp$ for any $a\in G$. To see this, note that if $a\in G$ then $\varphi(a,G)=G\backslash Xa\inv$ if $a\in X$, and $\varphi(a,G)=Xa\inv$ if $a\not\in X$.
\vspace{3pt}

Part $(4)$ follows easily from parts $(1)$ through $(3)$.
\end{proof}

\begin{corollary}\label{cor:fingen}
Let $\cB$ be a left-invariant sub-algebra of $\cB^{\st}_G$. The following are equivalent.
\begin{enumerate}[$(i)$]
\item $G^0_{\!\cB}$ has finite index in $G$.
\item $\Sg(\cB)$ is finite.
\item $G^0_{\!\cB^\sharp}$ has finite index in $G$.
\item $\Sg(\cB^\sharp)$ is finite.
\end{enumerate}
\end{corollary}
\begin{proof}
$(i)\Rightarrow (iii)\Rightarrow (iv)$ follow from Theorem \ref{thm:local}$(3)$ and  Corollary \ref{cor:projlim}. 

$(iv)\Rightarrow (ii)$. By Corollary \ref{cor:gentype}, restriction from $\Sg(\cB^\sharp)$ to $\Sg(\cB)$ is surjective.

$(ii)\Rightarrow (i)$. Suppose $\Sg(\cB)$ is finite and fix $p\in\Sg(\cB)$. Then $\{p\}$ is open in $\Sg(\cB)$ and so, by Proposition \ref{prop:LItop}, there is some $H\leq^{\cB}_{\fin} G$  such that $p$ is the unique type in $\Sg(\cB)$ containing $pH$. Without loss of generality, assume $pH=H$. Toward a contradiction, suppose $G^0_{\!\cB}\neq H$. Then some $K\leq^{\cB}_{\fin} G$ is a proper subgroup of $H$. Fix $g\in H\backslash K$. By Corollary \ref{cor:gentype}, there are $q,r\in\Sg(\cB)$ such that $K\in q$ and $gK\in r$. So $H\in q$, $H\in r$, and $q\neq r$, which contradicts the uniqueness of $p$. 
\end{proof}

We now focus on finding concrete examples of left-invariant Boolean algebras $\cB\seq\cB^{\st}_G$ such that $G^0_{\!\cB}$ has finite index in $G$. The following is a combinatorial version of \cite[Definition 5.13]{HrPiGLF}.

\begin{definition}\label{def:sif}
Let $V$ be a set. A relation $\varphi(x,y)$ on $G\times V$ is \textbf{left-invariant}  if, for any $g\in G$ and $b\in V$, there is some $c\in V$ such that $g\varphi(G,b)=\varphi(G,c)$.
\end{definition}

\begin{remark}\label{rem:sif}
The canonical example of a left-invariant relation is where $V=G$ and $\varphi(x,y)$ is defined by $x\in yA$ for some fixed $A\seq G$. More generally, this a natural setting for the study of a group $G$ definable in some first-order structure $M$. In this case, one often considers invariant formulas $\varphi(x,y)$ where the variable $x$ concentrates on $G$ and the variable $y$ is from an arbitrary sort in $M$ (e.g., $V=M^y$). 
\end{remark}

\begin{theorem}\label{thm:sif}
Suppose $\varphi(x,y)$ is a left-invariant stable relation on $G\times V$. Then $\cB_\varphi$ is a left-invariant sub-algebra of $\cB^{\st}_G$, and $G^0_{\!\cB_\varphi}$ has finite index in $G$.
\end{theorem}
\begin{proof}
Let $\cB=\cB_\varphi$ be the Boolean algebra generated by $\{\varphi(G,b):b\in V\}$. It is straightforward to show that $\cB$ is left-invariant and contained in $\cB^{\st}_G$. By Theorem \ref{thm:summeas}, there are $p_n\in S(\cB)$ and $\alpha_n\in [0,1]$, for $n\in\N$, such that  $\mu{\upharpoonright}\cB=\sum_{n=0}^\infty\alpha_n p_n$. Fix $n\in\N$ such that $\alpha_n>0$, and let $p=p_n$. For any $H\leq^{\cB}_{\fin} G$, we have $\mu(pH)\geq\alpha_n$, which implies that $[G:H]\leq \alpha_n\inv$. So we may choose some $H\leq^{\cB}_{\fin}G$ of maximal index in $G$. It follows that $H=G^0_{\!\cB}$. 
\end{proof}

\begin{remark}
Gannon's proof of Theorem \ref{thm:summeas} in \cite{GanStab} uses the L\"{o}wenheim-Skolem Theorem to reduce to the case of countable structures. We can do the same and give another proof of Theorem \ref{thm:sif}. In particular, suppose $\varphi(x,y)$ is a left-invariant stable relation on $G\times V$. By Corollary \ref{cor:fingen}, it suffices to show $\Sg(\cB_{\varphi})$ is finite. So suppose $\Sg(\cB_{\varphi})$ is infinite. Without loss of generality, we may assume $G$ is countable (view $(G,V)$ as a two-sorted structure in the group language with a predicate for $\varphi(x,y)$, and apply L\"{o}weinheim-Skolem to obtain a countable elementary substructure containing parameters for instances of $\varphi(x,y)$ that distinguish infinitely many generic $\varphi$-types). Now, the map $p\mapsto d_\varphi p$ from $S(\cB_\varphi)$ to $\cP(V)$ is injective by construction, and has image in $\cB^*_{\varphi}$ by  Theorem \ref{thm:deftypes}$(a)$. Since $G$ is countable, $\cB^*_{\varphi}$ is countably generated, and thus countable. Altogether, $\Sg(\cB_\varphi)$ is a countably infinite compact Hausdorff homogeneous space. But such spaces do not exist.

We also note that if $\varphi(x,y)$  is actually $k$-stable  for some $k\geq 1$, then $\Sg(\cB_\varphi)$ is finite by a result of Hrushovski and Pillay (see \cite[Lemma 5.16]{HrPiGLF}). 
\end{remark}

As an application, we now give an explicit  structure statement for stable subsets of groups, which is formulated entirely using genericity.

\begin{corollary}\label{cor:genstructure}
Suppose $A\seq G$ is stable, and set 
\[
H=\{g\in G:gxA\smd xA\text{ is not generic for any $x\in G$}\}.
\]
\begin{enumerate}[$(a)$]
\item  $H$ is a finite-index normal subgroup of $G$, and is in the Boolean algebra generated by $\{gAh:g,h\in G\}$. 
\item For any coset $C$ of $H$, exactly one of $C\cap A$ or $C\backslash A$  is generic. Moreover, there is a set $Y\seq G$, which is a union of cosets of $H$, such that $A\smd Y$ is not generic.
\item If $A$ is generic then $H\seq AA\inv\cap A\inv A$.
\end{enumerate}
\end{corollary} 
\begin{proof}
Let $\cB$ be the Boolean algebra generated by $\{gA:g\in A\}$. Then $\cB^\sharp$ is the Boolean algebra generated by $\{gAh:g,h\in A\}$. Note also that if we define $\varphi(x,y)$ on $G\times G$ such that $x\in yA$, then $\cB=\cB_{\varphi}$. Therefore, $G^0_{\cB^\sharp}$ has finite index in $G$ by Corollary \ref{cor:fingen} and Theorem \ref{thm:sif} (via Remark \ref{rem:sif}).
Note that $H=\bigcap_{x\in G}\Stab_\mu(xA)$. By definition of $\cB$, it follows that $H=\bigcap_{X\in\cB}\Stab_\mu(X)$, and so $H=G^0_{\!\cB^\sharp}$ by Theorem \ref{thm:local}$(3)$. So parts $(a)$ and $(b)$ follow from  Theorem \ref{thm:local}. 
For part $(c)$, suppose $A$ is generic. So there is some $p\in\Sg(\cB^\sharp)$ containing $A$ by Corollary \ref{cor:gentype}. By Theorem \ref{thm:local}$(4)$, we have $H=\Stab(p)$. So if $g\in H$ then $gA\in gp=p$, and so $gA\cap A\in p$, hence $gA\cap A\neq\emptyset$. It follows that $H\seq AA\inv$. Moreover, if $g\in G$ then $pg\in\Sg(\cB^\sharp)$ and so, by Theorem \ref{thm:local}$(4)$, $pg=p$ if and only if $p$ and $pg$ contain the same coset of $H$. By normality of $H$, we have $pg=p$ if and only if $g\in H$. So $H\seq A\inv A$ by a similar argument. 
\end{proof}

Part $(c)$ of the previous corollary is reminiscent of results on stabilizers and simple groups of finite Morley rank (see also Corollary \ref{cor:coninv}), and was also motivated by \cite{MPW}. One might wonder why we did not try to apply Theorem \ref{thm:sif} directly to $\cB^\sharp$ in the proof by choosing a different relation. More generally, if $\varphi(x,y)$ on $G\times V$ is left-invariant then $(\cB_\varphi)^\sharp=\cB_{\varphi^\sharp}$, where $\varphi^\sharp(x;y,z)$ is the relation on $G\times (V\times G)$ defined by $\varphi(xz,y)$. So if $\varphi(x,y)$ is stable then $\cB_{\varphi^\sharp}$ is a bi-invariant sub-algebra of $\cB^{\st}_G$. However, $\varphi^\sharp(x;y,z)$ may be unstable, as we see in the following examples. 

\begin{example}\label{ex:bad}$~$
\begin{enumerate}[$(a)$] 
\item (from \cite{CPpfNIP}) Let $G=\Sym(\N)$, and let $H=\{\sigma\in G:\sigma(0)=0\}\leq G$. Then $\varphi(x,y):= ``x\in yH"$ is $2$-stable. Given $i\geq 1$, let $a_i$ be the transposition $(0~i)$ and let $b_i\in G$ be any permutation that fixes $j\geq 1$ if and only if $j\geq i$. Then $b_j\in a_iHa_i$ if and only if $j\geq i$, and so $\varphi^\sharp(x;y,z)$ is unstable.
\item Let $G=\GL_2(\C)$, and let $H=\langle \begin{psmallmatrix}1 & 1\\ 0 & 1\end{psmallmatrix}\rangle =\{\begin{psmallmatrix}1 & n\\ 0 & 1\end{psmallmatrix}:n\in\Z\}\leq G$.
Then $\varphi(x,y):= ``x\in yH"$ is $2$-stable.   
Let $(p_i)_{i=1}^\infty$ be an increasing enumeration of the primes and, for $j\geq 1$, set $t_j=\prod_{i\leq j}p_i$. Then $\begin{psmallmatrix}1 & t_j\\ 0 & 1\end{psmallmatrix} \in \begin{psmallmatrix}p_i & 0\\ 0 & 1\end{psmallmatrix}H\begin{psmallmatrix}p_i & 0\\ 0 & 1\end{psmallmatrix}\inv$ if and only if $i\leq j$, and so $\varphi^\sharp(x;y,z)$ is unstable. 
\end{enumerate}
In both examples, one can also show that $\varphi^\sharp(x;y,z)$ has the \emph{independence property} and the \emph{strict order property} in the structure $(G,\cdot,H)$.
\end{example}

\begin{remark}
 The use of Theorem \ref{thm:sif} (and L\"{o}wenheim-Skolem in particular) can be avoided in the proof of Corollary \ref{cor:genstructure}. In fact, given a finite set $\cA\seq\cB^{\st}_G$, if $\cB$ is the Boolean algebra generated by $\{gA:g\in G,~A\in\cA\}$, then one can prove that  $G^0_{\!\cB}$ has finite index in $G$ as follows.  By Corollary \ref{cor:structure}, there is some $H\leq^{\cB}_{\fin} G$ such that, for any $A\in\cA$, there is a union $Y$ of left cosets of $H$ such that $\mu(A\smd Y)=0$. By induction, one can show that the same is true for any set $\cB$. Now, if $H\neq G^0_{\!\cB}$, then there is some $K\leq^{\cB}_{\fin} G$ which is a proper subgroup of $H$. So $\mu(K\smd Y)=0$ for some union $Y$ of left cosets of $H$, which is a contradiction. 
\end{remark}

\section{Remarks on stable additive combinatorics}\label{sec:SAC}

  We say that a group $G$ is \textbf{amenable} if there is a left-invariant probability measure on $\cP(G)$. A fundamental result is that $G$ is amenable if and only if it admits a \emph{F{\o}lner net}, i.e., a net $\mathcal{F}=(F_i)_{i\in I}$ of nonempty finite subsets of $G$ such that, for any $a\in G$, $|aF_i\cap F_i|/|F_i|\to 1$. For example, in $(\Z,+)$ any sequence of intervals of diverging length is a F{\o}lner net. Given $A\seq G$, we let $\delta^{\mathcal{F}}(A)=\limsup_{i\in I}|A\cap F_i|/|F_i|$ and $\delta_{\mathcal{F}}(A)=\liminf_{i\in I}|A\cap F_i|/|F_i|$. The \emph{upper} and \emph{lower} \emph{Banach density} of a subset $A$ of an amenable group $G$ are (respectively)
\begin{align*}
\overline{\delta}(A) &= \sup\{\delta^{\mathcal{F}}(A):\text{$\mathcal{F}$ is a F{\o}lner net for $G$}\},\text{ and}\\
\underline{\delta}(A) &= \inf\{\delta_{\mathcal{F}}(A):\text{$\mathcal{F}$ is a F{\o}lner net for $G$}\}.
\end{align*}
A good exercise is to show that a subset $A$ of an amenable group $G$ is generic if and only if $\underline{\delta}(A)>0$. 
In general, one has $\underline{\delta}(A)\leq\overline{\delta}(A)$, and this inequality may be strict. However, for stable sets this does not happen.

\begin{lemma}\label{lem:amen}
Suppose $G$ is an amenable group and $A\in \cB^{\st}_G$. Then $\overline{\delta}(A)=\underline{\delta}(A)$, and this density is rational. In particular, if $\overline{\delta}(A)>0$ then $A$ is generic.
\end{lemma}
\begin{proof}
Let $\alpha_1=\overline{\delta}(A)$ and $\alpha_0=\underline{\delta}(A)$. Fix $t\in\{0,1\}$. Then there is a F{\o}lner net $(F^t_i)_{i\in I}$ such that $|F^t_i\cap A|/|F^t_i|\to\alpha_t$ (see \cite[Theorem 4.16]{HiStDAS}). Let $\mu_t$ be a nonprincipal ultralimit of counting measures normalized on $F^t_i$. Then $\mu_t$ is a left-invariant  probability measure  on $\cP(G)$, and $\mu_t(A)=\alpha_t$. So $\mu_0(A)=\mu_1(A)$ by Theorem \ref{thm:intro}$(a)$, and this value is rational by Theorem \ref{thm:intro}$(c)$.
\end{proof}

It was shown in \cite{CoSS} that if $A\seq\N$ is definable in a (globally) stable expansion of $(\Z,+)$, then $A$ has upper Banach density $0$. Since no subset of $\N$ is generic (in $\Z$), the previous lemma generalizes this result to the ``local and in a model" case.

\begin{corollary}
If $A\seq\N$ is stable in $(\Z,+)$, then it has upper Banach density $0$.
\end{corollary}

The next proposition is motivated by \emph{Erd\H{o}s's sumset conjecture}, which says that if $A\seq\Z$ has positive upper Banach density, then there are infinite $B,C\seq\Z$ such that $B+C\seq A$ (this was recently proved in \cite{MoRiRo}). In \cite{ACG}, there is a short proof that if $G$ is a countable amenable group and $A\in\cB^{\st}_G$ has positive upper Banach density, then there are infinite $B,C\seq G$ such that $BC\seq A$.  Together with Lemma \ref{lem:amen}, the next proposition gives a different proof this result, which works for any amenable group and yields a much stronger conclusion. 

\begin{proposition}
Let $G$ be a group and suppose $A\in\cB^{\st}_G$ is generic. Then there is a finite set $F\seq G$ such that, for any infinite $B,C\seq G$, there are $g\in F$ and infinite $B'\seq B$ and $C'\seq C$ such that $B'C'\seq gA$.
\end{proposition}
\begin{proof}
Fix a finite set $F\seq G$ such that $G=FA$. Fix $B=\{b_n\}_{n=0}^\infty\seq G$ and $C=\{c_n\}_{n=0}^\infty\seq G$. Let $P=\{(i,j)\in\N\times\N:i<j\}$. Given $(i,j)\in P$, choose some $g_{i,j}\in F$ such that $b_ic_j\in g_{i,j}A$. Define $f\colon P\to F\times\{0,1\}$ such that $f(i,j)=(g_{i,j},0)$ if $b_jc_i\not\in g_{i,j}A$, and $f(i,j)=(g_{i,j},1)$ if  $b_jc_i\in g_{i,j}A$. By Ramsey's Theorem, there is an infinite set $I\seq\N$, and some $(g,k)\in F\times\{0,1\}$, such that $f(i,j)=(g,k)$ for all $i,j\in I$ with $i<j$. Since $A$ is stable in $G$, we cannot have $k=0$. So $k=1$, and we have $b_ic_j\in gA$ for all distinct $i,j\in I$. Partition $I=I_1\cup I_2$ into two infinite sets, and let $B'=\{b_i:i\in I_1\}$ and $C'=\{c_i:i\in I_2\}$. Then $B'C'\seq gA$.
\end{proof}

\begin{remark}
The previous result does not hold without the assumption of stability. For example, let $B=\{2x:x\in\Z,~x\geq 0\}$ and $C=\{2x:x\in\Z,~x<0\}$. Then $A=B\cup (C+1)$ is generic in $\Z$, but there do not exist infinite $B'\seq B$ and $C'\seq C$ such that $B'+C'\seq A+t$ for some $t\in\Z$.
\end{remark}

We say that a subset $A$ of a group $G$ is:
\begin{enumerate}
\item \textbf{thick} if for any finite $F\seq G$ there is some $g\in G$ such that $Fg\seq A$.
\item  \textbf{weakly generic} if $FA$ is thick for some finite $F\seq G$;
\item  \textbf{supergeneric} if $\bigcap_{g\in F}gA$ is generic for any finite $F\seq G$.
\end{enumerate}
The first notion is standard in combinatorial number theory (where generic sets are called \emph{syndetic} and weakly generic sets are called \emph{piecewise syndetic}), the second is from \cite{NewTD}, and the third is from  \cite{PoZ}. It is not hard to show that $A\seq G$ is generic if and only if $G\backslash A$ is not thick, and $A\seq G$ is supergeneric if and only if $G\backslash A$ is not weakly generic. In particular, if a set is supergeneric then it is generic and thick, and if a set is generic or thick then it is weakly generic.  

In the model theoretic context, it was observed by Newelski and Petrykowski in \cite{NePe} (and later by Poizat in \cite{PoZ}) that ultrafilters of weakly generic sets always exist.  Indeed, weakly generic sets are \emph{partition regular}, i.e., if $A\cup B$ is weakly generic then $A$ or $B$ is weakly generic. This fact is well-known in combinatorial number theory, and was shown by Bergelson, Hindman, and McCutcheon \cite{BHM}, with origins in even earlier work of Brown \cite{BrownPS}. It also yields the following characterization of when genericity and weak genericity coincide (see, e.g., \cite[Lemma 1.5]{NePe}). 

\begin{fact}\label{fact:w=g}
Let $G$ be a group, and suppose $\cB\seq\cP(G)$ is a left-invariant Boolean algebra. The following are equivalent.
\begin{enumerate}[$(i)$]
\item For any $A\in\cB$, $A$ is generic if and only if it is weakly generic. 
\item For any $A\in\cB$, $A$ is supergeneric if and only if it is thick.
\item There is a generic type in $S(\cB)$.
\end{enumerate}
\end{fact}

Let $G$ be a group. 
By Theorem \ref{thm:left-inv}, the conditions in Fact \ref{fact:w=g} hold for $\cB^{\st}_G$. Let $\mu$ be the unique bi-invariant probability measure on $\cB^{\st}_G$. Then  $A\in\cB^{\st}_G$ is supergeneric if and only if $G\backslash A$ is non-generic, and this also coincides with $\mu(A)=1$. Altogether, ``non-genericity" is a canonical notion of smallness for $\cB^{\st}_G$. Recall that $\cB^{\st}_G$ contains all subgroups of $G$ (see the proof of Corollary \ref{cor:stabCR}). The Boolean algebra generated by cosets of subgroups of $G$ is also called the ``coset ring" of $G$, and a more quantitative account of stability for the coset ring of an abelian group is given by Sanders in \cite{SandersSR}. When $G$ is abelian, its coset ring coincides with the \emph{Fourier algebra} of $G$, i.e., the algebra of subsets $A\seq G$ such that $1_A=\hat{\nu}$ (the \emph{Fourier-Stieljtes transform} of $\nu$) for some Borel measure $\nu$ on the compact group of characters on $G$ as a discrete group (see \cite[Theorem 3.1.3]{RudinFA}). Theorem \ref{thm:intro} says that, for arbitrary $G$,  $\cB^{\st}_G$ is essentially controlled by the coset ring of $G$ ``up to small sets". By restricting $\mu$, we also obtain a bi-invariant  probability measure on the coset ring of any group $G$. In this way, we can view the existence of $\mu$ as an extension of the following classical result of B. H. Neumann (see \cite[Section 4]{NeuBH}).

\begin{proposition}
Let $G$ be a group and suppose $G=C_1\cup\ldots\cup C_n$, where each $C_i$ is a coset of a subgroup $H_i\leq G$. Let $I$ be the set of $i\leq n$ such that $H_i$ has finite index in $G$. Then $G=\bigcup_{i\in I} C_i$ and $[G:H_i]\leq |I|$ for some $i\in I$.
\end{proposition}
\begin{proof}
We have $1=\mu(G)=\sum_{i=1}^n \mu(C_i)=\sum_{i\in I}\mu(C_i)=\sum_{i\in I}1/[G:H_i]$, which immediately implies the latter claim. Without loss of generality, each $C_i$ is a \emph{left} coset of $H_i$ (note that any right coset of $H_i$ is a left coset of some conjugate of $H_i$). Let $A=\bigcup_{i\in I} C_i$. Then $\mu(G\backslash A)=0$, and so $G=A$ since $A$ is a union of left cosets of the finite-index subgroup $\bigcap_{i\in I} H_i$.
\end{proof}

In model theory, a definable group is called \emph{definably connected} if it has no definable finite-index subgroups. We have a natural analogue of this notion in the local setting.
Given a group $G$ and a left-invariant Boolean algebra $\cB\seq\cP(G)$, we say $G$ is \textbf{$\cB$-connected} if no proper finite-index subgroup of $G$ is in $\cB$ (i.e., $G^0_{\!\cB}=G$). For example, $G$ is $\cB^{\st}_G$-connected if and only if it has no proper finite-index subgroups (equivalently,  the profinite completion $\hat{G}$ of $G$ is trivial). Examples of $\cB^{\st}_G$-connected groups include divisible groups and infinite simple groups.

\begin{corollary}\label{cor:connected}
Let $G$ be a group and suppose $\cB$ is a left-invariant sub-algebra of $\cB^{\st}_G$. The following are equivalent.
\begin{enumerate}[$(i)$]
\item $G$ is $\cB$-connected.
\item $G$ is $\cB^\sharp$-connected.
\item There is a unique generic type in $S(\cB)$.
\item If $A\in\cB$ then exactly one of $A$  or $G\backslash A$ is supergeneric.
\item If $A\in\cB$ then $A$ is generic if and only if it is supergeneric.
\item Then unique left-invariant measure on $\cB$ is $\{0,1\}$-valued.
\end{enumerate}
\end{corollary}
\begin{proof}
$(i)\Rightarrow (vi)$ follows from Corollary \ref{cor:structure}, and $(vi)\Rightarrow (v)$ follows from Fact \ref{fact:w=g} (and the subsequent discussion). For $(v)\Rightarrow (i)$, note that any proper finite-index subgroup of $G$ in $\cB$ is generic and not supergeneric.     We have $(i)\Rightarrow (iii)$ by Theorem \ref{thm:local}$(1)$, $(iii)\Rightarrow (ii)$ by Theorem \ref{thm:local}$(3)$, and $(ii)\Rightarrow (i)$ is trivial. Finally, $(iv)$ and $(v)$ are equivalent by Theorem \ref{thm:gensets}$(b)$.
\end{proof}

A classical result from the model theory of groups is that if $G$ is a definably connected group definable in a stable theory, then $G=AB$ for any definable generic subsets $A,B\seq G$. Let us prove this in our general setting. 

\begin{corollary}\label{cor:coninv}
Let $G$ be a group, and suppose $\cB$ is a left-invariant sub-algebra of $\cB^{\st}_G$ such that $G$ is $\cB$-connected. Then $G=AB\inv$ for any generic $A,B\in\cB$.
\end{corollary}
\begin{proof}
Suppose $A,B\in \cB$ are generic, and fix $g\in G$. Let $p\in \Sg(\cB)$ be the unique generic type. Then $A,gB\in p$, and so $A\cap gB\neq\emptyset$, which implies $g\in AB\inv$.
\end{proof}

\subsection*{Acknowledgements} I would like to thank Anand Pillay for directing me to the work of Ellis and Nerurkar in \cite{EllNer}, and also Kyle Gannon, Jason Long, and Joe Zielinski for helpful conversations. Some parts  of Section \ref{sec:SAC} relate to discussions with Artem Chernikov, James Freitag, Isaac Goldbring, and Frank Wagner at the 2017 AIM workshop ``Nonstandard methods in combinatorial number theory". I am also indebted to the referee for their comments and suggestions, which led to significant improvement to the exposition.

\bibliographystyle{amsplain}

\begin{thebibliography}{10}

\bibitem{ACG}
Uri Andrews, Gabriel Conant, and Isaac Goldbring, \emph{Definable sets
  containing productsets in expansions of groups}, J. Group Theory \textbf{22}
  (2019), no.~1, 63--82. \MR{3895638}



\bibitem{Ausbook}
Joseph Auslander, \emph{Minimal flows and their extensions}, North-Holland
  Mathematics Studies, vol. 153, North-Holland Publishing Co., Amsterdam, 1988,
  Notas de Matem\'{a}tica [Mathematical Notes], 122. \MR{956049}

\bibitem{BYGro}
Ita\"{i} Ben~Yaacov, \emph{Model theoretic stability and definability of types,
  after {A}. {G}rothendieck}, Bull. Symb. Log. \textbf{20} (2014), no.~4,
  491--496. \MR{3294276}

\bibitem{BHM}
Vitaly Bergelson, Neil Hindman, and Randall McCutcheon, \emph{Notions of size
  and combinatorial properties of quotient sets in semigroups}, Proceedings of
  the 1998 {T}opology and {D}ynamics {C}onference ({F}airfax, {VA}), vol.~23,
  1998, pp.~23--60. \MR{1743799}

\bibitem{BrownPS}
T.~C. Brown, \emph{An interesting combinatorial method in the theory of locally
  finite semigroups}, Pacific J. Math. \textbf{36} (1971), 285--289.
  \MR{0280627}

\bibitem{CoSS}
Gabriel Conant, \emph{Stability and sparsity in sets of natural numbers},
  Israel J. Math. \textbf{230} (2019), no.~1, 471--508. \MR{3941155}



\bibitem{CPpfNIP}
Gabriel Conant and Anand Pillay, \emph{Pseudofinite groups and {VC}-dimension}, to appear in J. Math. Log.,
  arXiv:1802.03361.

\bibitem{CPT}
Gabriel Conant, Anand Pillay, and Caroline Terry, \emph{A group version of
  stable regularity}, arXiv:1710.06309, accepted to Mathematical Proceedings of
  the Cambridge Philosophical Society.

\bibitem{EllNer}
R.~Ellis and M.~Nerurkar, \emph{Weakly almost periodic flows}, Trans. Amer.
  Math. Soc. \textbf{313} (1989), no.~1, 103--119. \MR{930084}

\bibitem{EllisLCTG}
Robert Ellis, \emph{Locally compact transformation groups}, Duke Math. J.
  \textbf{24} (1957), 119--125. \MR{0088674}

\bibitem{EllLTD}
\bysame, \emph{Lectures on topological dynamics}, W. A. Benjamin, Inc., New
  York, 1969. \MR{0267561}

\bibitem{Fremv4}
D.~H. Fremlin, \emph{Measure theory. {V}ol. 4}, Torres Fremlin, Colchester,
  2006, Topological measure spaces. Part I, II, Corrected second printing of
  the 2003 original. \MR{2462372}

\bibitem{GanStab}
Kyle Gannon, \emph{Measures and stability in a model},
  \url{https://www3.nd.edu/~kgannon1/Stability.pdf}, research note.

\bibitem{GroWAP}
A.~Grothendieck, \emph{Crit\`eres de compacit\'{e} dans les espaces
  fonctionnels g\'{e}n\'{e}raux}, Amer. J. Math. \textbf{74} (1952), 168--186.
  \MR{0047313}

\bibitem{HiStDAS}
Neil Hindman and Dona Strauss, \emph{Density in arbitrary semigroups},
  Semigroup Forum \textbf{73} (2006), no.~2, 273--300. \MR{2280825}

\bibitem{HiStbook}
\bysame, \emph{Algebra in the {S}tone-\v{C}ech compactification}, De Gruyter
  Textbook, Walter de Gruyter \& Co., Berlin, 2012, Theory and applications,
  Second revised and extended edition [of MR1642231]. \MR{2893605}
  
\bibitem{HofMo3}
Karl~H. Hofmann and Sidney~A. Morris, \emph{The structure of compact groups},
  De Gruyter Studies in Mathematics, vol.~25, De Gruyter, Berlin, 2013, A
  primer for the student---a handbook for the expert, Third edition, revised
  and augmented. \MR{3114697}

\bibitem{HrMLFF}
Ehud Hrushovski, \emph{The {M}ordell-{L}ang conjecture for function fields}, J.
  Amer. Math. Soc. \textbf{9} (1996), no.~3, 667--690. \MR{1333294}

\bibitem{HruAG}
\bysame, \emph{Stable group theory and approximate subgroups}, J. Amer. Math.
  Soc. \textbf{25} (2012), no.~1, 189--243. \MR{2833482}



\bibitem{HKPda}
E.~Hrushovski, K.~Krupi{\'n}ski, and A.~Pillay, \emph{Amenability, connected components, and definable actions}, arXiv:1901.02859, 2019.

\bibitem{HrPi1B}
E.~Hrushovski and A.~Pillay, \emph{Weakly normal groups}, Logic colloquium '85
  ({O}rsay, 1985), Stud. Logic Found. Math., vol. 122, North-Holland,
  Amsterdam, 1987, pp.~233--244. \MR{895647}

\bibitem{HrPiGLF}
\bysame, \emph{Groups definable in local fields and
  pseudo-finite fields}, Israel J. Math. \textbf{85} (1994), no.~1-3, 203--262.
  \MR{1264346}

\bibitem{Keis}
H.~Jerome Keisler, \emph{Measures and forking}, Ann. Pure Appl. Logic
  \textbf{34} (1987), no.~2, 119--169. \MR{890599}

\bibitem{MPW}
Amador Martin-Pizarro, Daniel Palac\'{i}n, and Julia Wolf, \emph{A
  model-theoretic note on the {F}reiman-{R}uzsa theorem}, arXiv:1912.02883,
  2019.

\bibitem{MoRiRo}
Joel Moreira, Florian~K. Richter, and Donald Robertson, \emph{A proof of a
  sumset conjecture of {E}rd{\H{o}}s}, Ann. of Math. (2) \textbf{189} (2019),
  no.~2, 605--652. \MR{3919363}
  
 \bibitem{NeuBH}
B.~H. Neumann, \emph{Groups covered by permutable subsets}, J. London Math.
  Soc. \textbf{29} (1954), 236--248. \MR{62122}

\bibitem{NewTD}
Ludomir Newelski, \emph{Topological dynamics of definable group actions}, J.
  Symbolic Logic \textbf{74} (2009), no.~1, 50--72. \MR{2499420}

\bibitem{NePe}
Ludomir Newelski and Marcin Petrykowski, \emph{Weak generic types and coverings
  of groups. {I}}, Fund. Math. \textbf{191} (2006), no.~3, 201--225.
  \MR{2278623}

\bibitem{PilDTH}
Anand Pillay, \emph{Dimension theory and homogeneity for elementary extensions
  of a model}, J. Symbolic Logic \textbf{47} (1982), no.~1, 147--160.
  \MR{644760}

\bibitem{PilGro}
\bysame, \emph{Generic stability and {G}rothendieck}, South Amer. J. Log.
  \textbf{2} (2016), no.~2, 437--442. \MR{3671044}

\bibitem{PiDR}
\bysame, \emph{Domination and regularity}, arXiv:1806.08806, 2018.

\bibitem{PoStG}
Bruno Poizat, \emph{Stable groups}, Mathematical Surveys and Monographs,
  vol.~87, American Mathematical Society, Providence, RI, 2001, Translated from
  the 1987 French original by Moses Gabriel Klein. \MR{1827833 (2002a:03067)}

\bibitem{PoZ}
\bysame, \emph{Superg\'en\'erix}, J. Algebra \textbf{404} (2014), 240--270,
  {\`A} la m{\'e}moire d'{\'E}ric Jaligot. [In memoriam {\'E}ric Jaligot].
  \MR{3177894}

\bibitem{RZbook}
Luis Ribes and Pavel Zalesskii, \emph{Profinite groups}, second ed., Ergebnisse
  der Mathematik und ihrer Grenzgebiete. 3. Folge. A Series of Modern Surveys
  in Mathematics [Results in Mathematics and Related Areas. 3rd Series. A
  Series of Modern Surveys in Mathematics], vol.~40, Springer-Verlag, Berlin,
  2010. \MR{2599132}
  
  \bibitem{RudinFA}
Walter Rudin, \emph{Fourier analysis on groups}, Wiley Classics Library, John
  Wiley \& Sons, Inc., New York, 1990, Reprint of the 1962 original, A
  Wiley-Interscience Publication. \MR{1038803}
  

\bibitem{SandersSR}
Tom Sanders, \emph{The coset and stability rings}, arXiv:1810.10461, 2018.

\bibitem{Shbook}
Saharon Shelah, \emph{Classification theory and the number of nonisomorphic
  models}, second ed., Studies in Logic and the Foundations of Mathematics,
  vol.~92, North-Holland Publishing Co., Amsterdam, 1990. \MR{1083551
  (91k:03085)}

\bibitem{StarGro}
Sergei Starchenko, \emph{On {G}rothendieck's approach to stability},
  \url{https://www3.nd.edu/~sstarche/papers/groth-stab.pdf}, research note.

\bibitem{TeWo}
C.~Terry and J.~Wolf, \emph{Stable arithmetic regularity in the finite field
  model}, Bull. Lond. Math. Soc. \textbf{51} (2019), no.~1, 70--88.
  \MR{3919562}

\end{thebibliography}

\end{document}